\documentclass[11pt,reqno]{amsart}
\usepackage{amssymb}
\usepackage{mathptmx}
\usepackage{mathtools}
\usepackage{mathrsfs}
\usepackage[export]{adjustbox}
\usepackage[all]{xy}
\usepackage{stmaryrd}
\usepackage{fancyhdr}
\usepackage{hyperref}
\usepackage{tikz-cd}
\usepackage{cite}
\usepackage{amsmath,amsfonts}
\usepackage{graphicx}
\usepackage{wrapfig}
\usepackage{array}
\usepackage{amsthm}
\usepackage[left=1.2in, right=1.2in, top=1in, bottom=1in]{geometry}
\usepackage{etoolbox}
\patchcmd{\section}{\scshape}{\bfseries}{}{}
\makeatletter
\renewcommand{\@secnumfont}{\bfseries}
\makeatother
\xyoption{all}
\DeclareMathOperator{\Hom}{Hom}
\DeclareMathOperator{\Spec}{Spec}

\DeclareMathOperator{\Mod}{Mod}

\newcommand{\mycirc}[1][black]{\Large\textcolor{#1}{\ensuremath\bullet}}

\newcommand{\nc}{\newcommand}   
\nc{\mc}{\mathcal}              
\nc{\on}{\operatorname}         
\nc{\wt}{\widetilde}            

\nc{\ses}[3]{{#1}\hookrightarrow{#2}\twoheadrightarrow{#3}}

\nc{\MS}{\mathbf{Mat}_{\bullet}}

\nc{\Set}{\mathcal{S}et_{\bullet}}

\nc{\E}{\mc{E}}                 
\nc{\EE}{\mathfrak{E}}          
\nc{\MM}{\mathfrak{M}}          

\nc{\FF}{\mathbb{F}}            
\nc{\GG}{\mathbb{G}}            

\nc{\C}{\mathcal{C}}            
\nc{\fl}{\mathbf{FL}}           

\newcommand{\angles}[1]{\left\langle #1 \right\rangle}

\input xy
\xyoption{all}
\theoremstyle{definition}
\newtheorem{mydef}{\textbf{Definition}}[section]
\newtheorem{myeg}[mydef]{\textbf{Example}}

\newtheorem{question}[mydef]{\textbf{Question}}
\newtheorem{rmk}[mydef]{\textbf{Remark}}

\theoremstyle{plain}
\newtheorem{mythm}[mydef]{\textbf{Theorem}}

\newtheorem*{nothma}{\textbf{Theorem A}}
\newtheorem*{nothmb}{\textbf{Theorem B}}
\newtheorem*{nothmc}{\textbf{Theorem C}}
\newtheorem*{nothmd}{\textbf{Theorem D}}

\newtheorem{lem}[mydef]{\textbf{Lemma}}
\newtheorem{pro}[mydef]{\textbf{Proposition}}

\newtheorem{cor}[mydef]{\textbf{Corollary}}

\begin{document}

\title{Proto-exact categories of modules over semirings and hyperrings}
\author{Jaiung Jun}
\address{Department of Mathematics, State University of New York at New Paltz, NY 12561, USA}
\email{junj@newpaltz.edu}

\author{Matt Szczesny}
\address{Department of Mathematics and Statistics, Boston University, 111 Cummington Mall, Boston, USA}
\email{szczesny@math.bu.edu}

\author{Jeffrey Tolliver}
\address{}
\email{jeff.tolli@gmail.com}
\makeatletter
\@namedef{subjclassname@2020}{%
	\textup{2020} Mathematics Subject Classification}
\makeatother

\thanks{}

\begin{abstract}
\emph{Proto-exact categories}, introduced by Dyckerhoff and Kapranov, are a generalization of Quillen exact categories which  provide a framework for defining algebraic K-theory and Hall algebras in a \emph{non-additive} setting. This formalism is well-suited to the study of categories whose objects have strong combinatorial flavor. 

In this paper, we show that the categories of modules over semirings and hyperrings - algebraic structures which have gained prominence in tropical geometry - carry proto-exact structures. 

In the first part, we prove that the category of modules over a semiring is equipped with a proto-exact structure; modules over an idempotent semiring have a strong connection to matroids.  We also prove that the category of algebraic lattices $\mathcal{L}$ has a proto-exact structure, and furthermore that the subcategory of $\mathcal{L}$ consisting of finite lattices is equivalent to the category of finite $\mathbb{B}$-modules as proto-exact categories, where $\mathbb{B}$ is the \emph{Boolean semifield}. We also discuss some relations between $\mathcal{L}$ and geometric lattices (simple matroids) from this perspective. 

In the second part, we prove that the category of modules over a hyperring has a proto-exact structure. In the case of finite modules over the \emph{Krasner hyperfield} $\mathbb{K}$, a well-known relation between finite $\mathbb{K}$-modules and finite incidence geometries yields a combinatorial interpretation of exact sequences. 
\end{abstract}

\maketitle


\section{Introduction}

Recent years have seen several attempts to formulate notions of algebraic geometry in characteristic one based on monoids, semirings, hyperrings, and blueprints \cite{Deitmar, con1, con2, oliver1, soule2004varietes}. One reason for this  effort is the desire to develop scheme-theoretic foundations for tropical geometry \cite{giansiracusa2016equations, lorscheid2015scheme}.  In these ``exotic" theories, just as in ``ordinary" algebraic geometry, an affine scheme corresponds one of the aforementioned algebraic structures, which we generically denote $A$ for the moment. Proceeding by analogy, one expects a quasi-coherent sheaf on $\Spec(A)$ to correspond to some appropriate notion of $A$-module, and is therefore led to study the category A-mod of A-modules and its homological properties in general.  Here, the situation is complicated by the fact that unlike the case of commutative rings, A-mod is typically not abelian or even additive, and so what should be considered an ``exact sequence" does not have an immediate answer. 

\emph{Proto-exact categories}, introduced by T.~Dyckerhoff and M.~Kapranov \cite{dyckerhoff2012higher} as a generalization of Quillen exact cateogires, provide a flexible framework for exact sequences in (potentially) non-additive categories. Roughly speaking, a proto-exact category is a pointed category with two distinguished classes of morphisms (\emph{admissible monomorphisms} and \emph{admissible epimorphisms}) satisfying certain conditions on pullback and pushout diagrams from which one can obtain a notion of \emph{admissible exact sequences}.\footnote{See Definition \ref{definition: proto_exact} for the precise definition.} Several interesting ``combinatorial'' categories are equipped with a proto-exact structure, for instance, the category of matroids \cite{eppolito2018proto}, the category of representations over a quiver (and more generally any monoid) over ``the field with one element'' \cite{szczesny2012representations}, \cite{jun2020quiver}, \cite{jun2021coefficient}. Categories with more algebro-geometric flavors, which are not additive, have been explored in \cite{szczesny2018hopf}, \cite{jun2020toric}, \cite{eberhardt2020group}. 

There are at least two ``features" associated with a proto-exact category $\mathcal{A}$:

\begin{enumerate}
\item If $\mathcal{A}$ is \emph{finitary}, in the sense that the number $\vert \on{Ext}^1(X,Y) \vert$ of inequivalent short exact sequences is finite for each pair of objects $X, Y \in \mathcal{A}$, then one may define the \emph{Hall algebra} $H_{\mathcal{A}}$. This is an associative (and in ``good" cases Hopf) algebra spanned by the isomorphism classes of objects of $\mathcal{A}$, whose structure coefficients count the number of extensions between objects. Classically -  for instance when $\mathcal{A}$ is the category of quiver representations or coherent sheaves on a curve over over a finite field \cite{ringel1990hall, kapranov1997eisenstein}, $H_{\mathcal{A}}$ yields quantum groups and related objects and is an important tool in their representations theory. When $\mathcal{A}$ is non-additive, $H_{\mathcal{A}}$ typically has a combinatorial flavor.  In this case, the Hall algebra often becomes a Hopf algebra, where the product of two objects is obtained by ``assembling'' two objects into a new object and coproduct encodes all possible ways to ``disassemble'' the given object into two objects. As an application, one may study various operations and identities for combinatorial objects from the Hall algebra perspective.\footnote{See \cite{iovanov2019hopf} for this line of ideas.}
\vspace{0.2in}

\item One can define a well-behaved version of algebraic K-theory for $\mathcal{A}$ (either through Quillen's Q-construction or Waldhausen's S-construction - see \cite{dyckerhoff2012higher, eberhardt2020group, hekkingThesis}. Even for relatively simple combinatorial categories $\mathcal{A}$ this is a rich and interesting invariant \cite{chu2012sheaves, eppolito2018proto}. 

\end{enumerate}

\medskip

The main goal of the current paper is to enlarge the catalogue of non-additive proto-exact categories by showing these include the categories of modules over \emph{semirings} as well as \emph{hyperrings}. Modules over an idempotent semiring are closely related to matroid theory \cite{giansiracusa2018grassmann} and modules over a hyperring have an interesting connection to finite incidence geometries \cite{con4}, \cite{jun2018association} and matroids \cite{baker2016matroids}. We also examine the category of algebraic lattices in relation to finite modules over $\mathbb{B}$, and discuss how the proto-exact structure of algebraic lattices is related to the proto-exact structure of the category of matroids in \cite{eppolito2018proto} via geometric lattices.
Along the way, we also investigate whether the proto-exact categories constructed are finitary, and thus whether the Hall algebra $H_{\mathcal{A}}$ is defined. \\

Let $R$ be a semiring and $M$ be an $R$-module. We define admissible monomorphisms (resp.~admissible epimorphisms) to be equalizers (resp.~coequalizers) of some morphisms and the zero map. We first prove the following. 

\begin{nothma}[Theorem \ref{theorem: main theorem for modules}]
Let $R$ be a semiring. With the admissible monomorphisms and admissible epimorphisms as above, the category $\Mod_R$ of modules over $R$ is a proto-exact category.
\end{nothma}

Next, we turn our attention to the category of \emph{algebraic lattices} (not necessarily finite). Recall that an algebraic lattice is a complete lattice such that any element is a join of compact elements. In particular, any finite lattice is an algebraic lattice. For algebraic lattices, roughly we define admissible monomorphisms (resp.~admissible epimorphisms) to be downward closed subsets (resp.~upward closed subsets). See Definition \ref{definition: normal maps for lattices} for the precise definition. We prove the following. 

\begin{nothmb}[Theorem \ref{theorem: main theorem for lattices}]
With the admissible monomorphisms and admissible epimorphisms as above, the category $\mathcal{L}$ of algebraic lattices is a proto-exact category.
\end{nothmb}

Let $\mathcal{L}^c$ be the subcategory of $\mathcal{L}$ consisting of finite lattices. $\mathcal{L}^c$ is a proto-exact category with the induced proto-exact structure (from $\mathcal{L})$. Also, from Theorem $A$, one can prove that the category $\textrm{Mod}_\mathbb{B}^{\textrm{fin}}$ of finite $\mathbb{B}$-modules is proto-exact with the induced proto-exact structure (from $\Mod_\mathbb{B}$). We prove the following.

\begin{nothmc}[Corollary \ref{corollary: main theorem lattice and module}]
 $\emph{Mod}_\mathbb{B}^{\emph{fin}}$ is equivalent to $\mathcal{L}^c$ as proto-exact categories.
\end{nothmc}

A \emph{geometric lattice} is a finite semimodular lattice in which every element is a join of atoms. Geometric lattices provide another cryptomorphic definition for (simple) matroids. In fact, the subcategory $\mathcal{G}$ of $\mathcal{L}$ consisting of geometric lattices is a proto-exact subcategory of $\mathcal{L}$ (Proposition \ref{proposition: matroid geometric lattices}). Note that the category $\mathcal{G}$ contains ``more morphisms'' than the category of matroids with strong maps as in \cite{eppolito2018proto}. In particular, the proto-exact structure on $\mathcal{G}$ is different from the proto-exact structure for the category of matroids studied in \cite{eppolito2018proto}. For more details, see Section \ref{subsection: semiring further directions}. \\

Next, we move to \emph{hyperrings}. A morphism of hyperrings preserves multi-valued addition in a ``weak sense'' (see Definition \ref{definition: hypergroup and morphisms}). We define admissible monomorphisms (resp.~admissible epimorphisms) to be injective (resp.~surjective) morphisms which preserve multi-valued addition in a ``strong sense''. Then, we prove the following. 

\begin{nothmd}[Theorem \ref{theorem: main theorem for hyperrngs}]
Let $H$ be a hyperring. With the admissible monomorphisms and admissible epimorphisms as above, the category $\Mod_H$ of modules over $H$ is a proto-exact category.	
\end{nothmd}
Finally, we recall a well-known equivalence between finite modules over the Krasner hyperfield $\mathbb{K}$ (Example \ref{example: Krasner hyperfield}) and finite incidence geometries which could potentially allow us to describe the structure constants for the Hall algebra associated $\Mod_{\mathbb{K}}$ by using finite incidence geometries. This is the topic of a future paper. 
\bigskip

\noindent \textbf{Acknowledgments}  The authors would like to thank Chris Eppolito for his helpful comments on the first draft of the paper.


\section{Preliminaries} \label{section: preliminaries}

\subsection{Proto-exact categories}\label{subsection: proto-exact categories}

In \cite{ringel1990hall}, Ringel defined the Hall algebra of the category $\textrm{Rep}(Q,\mathbb{F}_q)$ of representations of a quiver $Q$ over a finite field $\mathbb{F}_q$, and proved that in the case of a simply laced Dynkin quiver $Q$, the associated Hall algebra is the upper triangular part of the quantum group which is classified by the same Dynkin diagram. Later in \cite{kapranov1997eisenstein}, Kapranov investigated $\textrm{Coh}(X)$, the category of coherent sheaves on a smooth projective curve $X$ over a finite field $\mathbb{F}_q$. and proved when $X$ is a projective line, some subalgebra of $H_{\textrm{Coh}(X)}$ is isomorphic to a positive part of the quantum affine algebra $U_q(\widehat{\mathfrak{sl}_2})$. 

Notice that both $\textrm{Rep}(Q,\mathbb{F}_q)$ and $\textrm{Coh}(X)$ are examples of a finitary abelian category, and one can indeed associate a Hall algebra $H_\mathcal{A}$ to \emph{any} finitary abelian category as follows: the underlying vector space of $H_\mathcal{A}$ is spanned by the set $\textrm{Iso}(\mathcal{A})$ of isomorphism classes of $\mathcal{A}$. For each $A, B \in \textrm{Iso}(\mathcal{A})$, multiplication is defined as follows:
\begin{equation}\label{eq: 1}
	A\cdot B := \sum_{C \in \textrm{Iso}(\mathcal{A})}\textbf{a}^C_{A,B}C,
\end{equation}
where 
\begin{equation}\label{eq: 2}
	\textbf{a}^C_{A,B}=\#\{D \subseteq C \mid D\simeq B\textrm{ and } C/D \simeq A\}.
\end{equation}
Then one linearly extends \eqref{eq: 1} to $H_\mathcal{A}$. If $\mathcal{A}$ satisfies some extra conditions, it becomes a Hopf algebra. 

As a first step into a non-additive setting, one may notice from \eqref{eq: 1} and \eqref{eq: 2} that to define the Hall algebra $H_\mathcal{A}$ of an abelian category, one does not need the full strength of the abelian categories; one only needs a notion of ``exact sequences'' (or proper definitions of ``subobjects'' and ``quotient objects''). For instance, Quillen exact categories would suffice to define the algebraic structures as in \eqref{eq: 1} and \eqref{eq: 2}. 

In \cite{dyckerhoff2012higher}, Dyckerhoff and Kapranov introduced the notions of proto-exact and proto-abelian categories which allow one to construct the Hall algebra in a more general setting beyond abelian categories. In particular, this allows one to define and study Hall algebras for a non-additive category in a systemic way. In the following, we recall basic definitions for proto-exact categories which will be used in sequel. 

\begin{mydef}\label{definition: proto_exact}
A proto-exact category is a pointed category $\C$ equipped with two distinguished classes of morphisms $\MM$ (admissible monomorphisms) and $\mathfrak{E}$ (admissible epimorphisms) which satisfy the following conditions.
\begin{enumerate}
\item 
For each object $A \in \C$, a morphism $0 \rightarrow A$ is in $\MM$, and a morphism $A \rightarrow 0$ is in $\EE$.
\item 
$\MM$ and $\EE$ contain all isomorphisms and are closed under composition.
\item 
A commutative square as in \eqref{eq: bi-cat}, with
		$i,i' \in \MM$ and $j, j' \in \EE$, is Cartesian (pullback) if and only if it is co-Cartesian (pushout).
\begin{equation}\label{eq: bi-cat}
\begin{tikzcd}
A
\ar[r,"i",hook]
\ar[d,"j",two heads,swap]
&
B
\ar[d,"j'",two heads]
\\
A'
\ar[r,"i'",hook]
&
B'
\end{tikzcd}
\end{equation}

\item 
A diagram as in \eqref{eq: pullback1}, with $i' \in \MM$ and $j' \in \EE$, can be completed to a
bi-Cartesian square \eqref{eq: bi-cat} with $i \in \MM$ and
$j \in \EE$.
\begin{equation}\label{eq: pullback1}
\begin{tikzcd}
	{}
	&
	B
	\ar[d,"j'",two heads]
	\\
	A'
	\ar[r,"i'",hook]
	&
	B'
\end{tikzcd}
\end{equation}
\item 
A diagram as in \eqref{eq: com2}, with $i \in \MM$ and $j \in \EE$, can be completed to a bi-Cartesian
square \eqref{eq: bi-cat} with $i' \in \MM$ and $j' \in \EE$.
\begin{equation}\label{eq: com2}
\begin{tikzcd}
	A
	\ar[r,"i",hook]
	\ar[d,"j",two heads,swap]
	&
	B
	\\
	A'
	&
	{}
\end{tikzcd}
\end{equation}
\end{enumerate}
\end{mydef}

Let $\mathcal{C}$ be a proto-exact category. By an \emph{admissible short exact sequence} in $\mathcal{C}$, we mean a bi-Cartesian square as in \eqref{eq: ses proto-exact} with $i \in \MM$ and $j \in \EE$.
\begin{equation}\label{eq: ses proto-exact}
	\begin{tikzcd}
		A
		\ar[r,"i",hook]
		\ar[d,two heads]
		&
		B
		\ar[d,"j",two heads]
		\\
		0
		\ar[r,hook]
		&
		C
	\end{tikzcd}
\end{equation}
We sometimes call the diagram \eqref{eq: ses proto-exact} an \emph{admissible extension of $C$ by $A$}, and it will also be denoted by
\begin{equation}\label{ses}
\ses{A}{B}{C}.
\end{equation}
When there is no potential confusion, we will simply call \eqref{ses} a short exact sequence in $\C$. In \eqref{ses} we denote the object $C$, which is unique up to a unique isomorphism, by $B/A$.

\begin{mydef}
Let $\C$ and $\mathcal{D}$ be proto-exact categories. 
\begin{enumerate}
	\item
A functor $F \colon \C \mapsto \mathcal{D}$ is said to be \emph{exact} if $F$ preserves admissible short exact
sequences.
\item 
In $\C$, two extensions $\ses{A}{B}{C}$ and $\ses{A}{B'}{C'}$ of
$C$ by $A$ are \emph{equivalent} if there is a commutative diagram
\begin{equation}
\begin{tikzcd}
	A
	\ar[r,hook]
	\ar["\on{id}",d,swap]
	&
	B
	\ar[r,two heads]
	\ar["\cong",d]
	&
	C
	\ar["\on{id}",d]
	\\
	A
	\ar[r,hook]
	&
	B'
	\ar[r,two heads]
	&
	C
\end{tikzcd}
\end{equation}
The set of equivalence classes of such sequences is denoted
$\on{Ext}_{\C}(C,A)$.  
\end{enumerate}
\end{mydef}	
	
The following is the finiteness condition that we need to ensure the existence of the Hall algebra of a proto-exact category.

\begin{mydef}\label{defi:finitary}
A proto-exact category $\C$ is {\em finitary} if, for every pair of objects $A$ and $B$, the sets $\on{Hom}_{\C}(A,B)$ and $\on{Ext}_{\C}(A,B)$ are finite sets.
\end{mydef}

\begin{myeg}\label{example: proto-exact}
The following are typical examples of proto-exact categories.
\begin{enumerate}
\item 
Any Quillen exact category is proto-exact with the same exact structure. 
\item 
A simple example of proto-exact categories (which is not an additive category) is the category $\textrm{\textbf{Set}}_\bullet$ of pointed sets; $\MM$ consists of all pointed injections, and $\EE$ consists of all pointed surjections
\[
f: (A,*_A) \rightarrow (B,*_B)
\]
such that $f \vert_{A \backslash f^{-1}(*_B)}$ is injective.  The full subcategory $\textrm{\textbf{Set}}_\bullet^{\textrm{fin}}$ of finite pointed sets is finitary. See \cite[Section 1.3]{dyckerhoff2018higher} for the Hall algebra of $\textrm{\textbf{Set}}_\bullet^{\textrm{fin}}$ in perspective of algebraic geometry over ``the field with one element'' along with other interesting examples. Also, the first and the second authors further explored certain categories of coherent sheaves arising from algebraic geometry over monoids in \cite{jun2020toric}.
\item 
In \cite{eppolito2018proto}, together with C.~Eppolito, the first and the second authors proved that the category $\textbf{Mat}_\bullet$ of pointed matroids is equipped with a finitary proto-exact structure under which the Hall algebra is isomorphic to the dual of a matroid-minor Hopf algebra. 
\end{enumerate}
\end{myeg}

Now, one can define the Hall algebra $H_\mathcal{C}$ of a finitary proto-exact category $\mathcal{C}$ as in the case of abelian categories; it has the underlying set as a vector space spanned by the set $\textrm{Iso}(\mathcal{C})$ of isomorphism classes of $\mathcal{C}$. For each $M,N \in \textrm{Iso}(\mathcal{C})$, multiplication is as in \eqref{eq: 1}, where $D \subseteq C$ is a subobject of $C$ and $C/D$ is the object (unique up to a unique isomorphism) such that the following is a short exact sequence in $\mathcal{C}$:
\begin{equation}
\ses{D}{C}{C/D}.
\end{equation}
As in the classical case, when $\mathcal{C}$ satisfies further conditions, $H_{\mathcal{C}}$ becomes a Hopf algebra. For instance, this is the case of matroids, finite pointed sets, and representations of a quiver over ``the field with one element''. See \cite{eppolito2018proto}, \cite{szczesny2012representations}, \cite{jun2020quiver}, \cite{jun2021coefficient}, \cite{szczesny2018hopf}.

\subsection{Semirings and hyperrings} \label{subsection: semirings and hyperrings}

\subsubsection{\textbf{Semirings}}

By a \emph{semiring}, we mean an nonempty set $R$ with two binary operations $+,\cdot$ such that $(R,+,1_R)$ and $(R,\cdot,0_R)$ are commutative monoids and two binary operations are compatible in the sense that $(a+b)\cdot c = a\cdot c + a\cdot b$ for any $a,b,c \in R$. When $(R\setminus \{0_R\},\cdot)$ is a group, $R$ is said to be a \emph{semifield}. By an $R$-module for a semiring $R$, we mean a monoid $M$ with the usual axioms. 

\begin{mydef}\label{definition: ideals for semirings}
Let $R$ be a semiring. 
\begin{enumerate}
\item 
Let $M$ be an $R$-module. A submodule $N$ is said to be \emph{saturated} if $x, x+y \in N$ implies $y \in N$ for all $x,y \in M$. 
\item 
A semiring $R$ is \emph{additively idempotent} if $a+a = a$ for all $a\in R$.	
	\end{enumerate}
\end{mydef}

\begin{myeg}\label{example: Boolean semifield}
Let $\mathbb{B}=\{0,1\}$. With the following addition and multiplication, $\mathbb{B}$ is a semifield.
\[
0\cdot 0=0, \quad 1\cdot 1=1, \quad 0\cdot 1=0, \quad 0+0=0, \quad 1+0=1, \quad 1+1=1. 
\]
$\mathbb{B}$ is said to be the \emph{Boolean semifield}. 
\end{myeg}

\begin{myeg}\label{example: tropical semifield}
Let $\mathbb{R}_{\max}=\mathbb{R}\cup \{-\infty\}$ be the set of real numbers together with the symbol $-\infty$. With the following addition $\oplus$ and multiplication $\odot$, $\mathbb{R}_{\max}$ is a semifield. 
\[
a\oplus b= \max\{a,b\}, \quad a\odot b=a+b,
\]
where $+$ is the usual addition of real numbers. $\mathbb{R}_{\max}$ is said to be the \emph{tropical semifield} and plays an important role in tropical geometry; it plays the role that is classically played by the complex numbers. 
\end{myeg}

The following example illustrates that idempotent semirings are neither restrictive nor exotic. 

\begin{myeg}
Let $K$ be a valued field whose value group is $\mathbb{R}$,  $\mathcal{O}_K$ the ring of integers, and $A$ a $K$-algebra. Let $\mathbb{S}_{\mathrm{fg}}(A)$ be the set of $\mathcal{O}_K$-submodules of $A$. Then $\mathbb{S}_{\mathrm{fg}}(A)$ is equipped with a natural idempotent semiring structure: for $N_1=\angles{S_1}, N_2=\angles{S_2} \in \mathbb{S}_{\mathrm{fg}}(A)$, we define addition and multiplication as follows:
\[
N_1+N_2:=\angles{S_1\cup S_2}, \quad N_1N_2:=\angles{S_1S_2}. 
\]
For more details in relation to tropical geometry, we refer the readers to \cite{jun2020vector}. This correspondence was also employed in \cite{jun2020lattices} to enrich Hochster's theorem on \emph{spectral spaces} in \cite{hochster1969prime} by showing that a topological space $X$ is spectral if and only if $X$ is the ``saturated prime spectrum'' of an idempotent semiring. In fact, in \cite{jun2020lattices}, with S.~Ray, the first and the third authors proved that the category of spectral spaces is equivalent to a certain subcategory of the category of idempotent semirings.
\end{myeg}

\begin{rmk}
In this paper, by an idempotent semiring we always mean an additively idempotent semiring. Equivalently, an idempotent semiring is an $\mathbb{B}$-algebra, where $\mathbb{B}$ is the Boolean semifield.
\end{rmk}

Let $M$ be an $R$-module and $N$ be a submodule of $M$. $N$ defines a congruence relation on $M$ as follows: for $x,y \in M$
\begin{equation}\label{eq: quo}
x \sim y \iff x+n=y+n' \textrm{ for some } n,n' \in N.
\end{equation}
The quotient $M/N$ is defined to be the set of equivalence classes under \eqref{eq: quo} with the induced operations. One can easily check that $M/N$ is an $R$-module.

\subsubsection{\textbf{Hyperrings}} Next, we recall the definition of hyperrings. Let $M$ be a nonempty set and $\mathcal{P}^*(M)$ be the set of all nonempty subsets of $M$. By a \emph{hyperaddition}, we mean a binary function $\phi:M \times M \to \mathcal{P}^*(M)$ such that $\phi(a,b)=\phi(b,a)$. We simply denote $\phi(a,b)$ by $a+b$ or $a+_Mb$. If $A,B \subseteq M$, we use the following notation:
\begin{equation}\label{eq: hypersum of sets}
A+B:=\bigcup_{a \in A,b\in B}(a+b).
\end{equation}

In what follows, we will identify an element $a$ and the singleton $\{a\}$.

\begin{mydef}\label{definition: hypergroup and morphisms}
A \emph{hypergroup} is a nonempty set $M$ with a hyperaddition satisfying the following:
\begin{enumerate}
\item 
For any $a, b, c \in M$, we have $(a+b)+c=a+(b+c)$.
	\item
There exists a unique element $0 \in M$ such that $a+0=a$ for all $a \in M$. 
	\item
For each $a \in M$, there exists a unique element $b$, denoted by $-a$, such that $0 \in a+b$.
	\item
For any $a,b,c \in M$, if $a \in b+c$, then $c \in a+(-b)$. 
\end{enumerate}
A morphism $f:M_1 \to M_2$ between hypergroups is a function such that $f(0)=0$ and for all $a,b \in M_1$,
\begin{equation}\label{eq:123}
f(a+b) \subseteq f(a)+f(b).
\end{equation}
A morphism $f$ is said to be \emph{strict} if \eqref{eq:123} is an equality. 
\end{mydef}

\begin{mydef}
By a \emph{hyperring}, we mean a hypergroup $H$ equipped with (usual) multiplication $\cdot$ satisfying the following two conditions:
\begin{enumerate}
	\item 
$(H,\cdot,1_H)$ is a commutative monoid, 
\item 
Hyperaddition and multiplication are compatible in the sense that $a\cdot(b+c)=a\cdot b +a\cdot c$.
\end{enumerate} 
If $(H\setminus \{0_H\},\cdot)$ is a group, then $H$ is said to be a hyperfield. A morphism $f:H_1 \to H_2$ between hyperrings is a morphism of hypergroups which also preserves multiplication. A morphism is said to be \emph{strict} if it is strict as a morphism of hypergroups. 
\end{mydef}

\begin{mydef}\label{definition: module over a hyperring}
Let $H$ be a hyperring. By an $H$-module, we mean a hypergroup $M$ together with a map $\phi:H \times M \to M$ such that (we denote $\phi(r,m)$ by $rm$.)
\begin{enumerate}
\item 
$1m=m$ for all $m \in M$, $0m=0$ for all $m \in M$, 
\item 
$(xy)m=x(ym)$ for all $x,y \in H$ and $m \in M$ and 
\item 
$x(m_1+m_2)=xm_1+xm_2$ for all $x \in H$ and $m_i \in M$, 
\item 
$(x+y)m = xm+ym$ for all $x,y \in H$ and $m \in M$.
\end{enumerate}
By a morphism of $H$-modules, we mean a morphism $f:M_1 \to M_2$ of hypergroups such that $f(rm)=rf(m)$ for all $r \in H$ and $m \in M_1$.
\end{mydef}

\begin{myeg}\label{example: Krasner hyperfield}
Let $\mathbb{K}=\{0,1\}$. With the following hyperaddition and multiplication $\mathbb{K}$ becomes a hyperfield. 
\[
1\cdot 1=1, \quad 1\cdot 0 =0, \quad 0\cdot 0=0, \quad 0+0=0 \quad 1+0=1, \quad 1+1=\mathbb{K}. 
\]
$\mathbb{K}$ is said to be the \emph{Krasner hyperfield}.
\end{myeg}

\begin{myeg}
Let $\mathbb{S}=\{-1,0,1\}$. With the following hyperaddition and multiplication $\mathbb{S}$ becomes a hyperfield. 
\[
1\cdot 1=(-1)\cdot(-1)=1, \quad 1\cdot (-1) =-1, \quad 0\cdot (-1)=0\cdot 0 = 0 \cdot 1=0, 
\]
\[
0+0=0, \quad 1+1=1+0=1, \quad (-1)+(-1)=(-1)+0=-1,\quad 1+(-1)=\mathbb{S}. 
\]
$\mathbb{S}$ is said to be the \emph{sign hyperfield}.
\end{myeg}

\begin{myeg}
Let $\mathbb{T}=\mathbb{R}\cup\{-\infty\}$, the set of real numbers together with the symbol $-\infty$. With the following hyperaddition $\oplus$ and multiplication $\odot$, $\mathbb{T}$ becomes a hyperfield.
\[
a\odot b = a+b, \quad a\oplus b= \begin{cases}
	\max\{a,b\} \textrm{ if $a \neq b$, }\\
	[-\infty,a] \textrm{ if $a=b$,  }
\end{cases}
\]
where $+$ is the usual addition of real numbers and $[-\infty,a]=\{x \in \mathbb{T} \mid x \leq a\}$. $\mathbb{T}$ is said to be the \emph{tropical hyperfield}. 
\end{myeg}

The following example shows that although hyperrings look exotic one can easily construct them from a commutative ring $A$ and a multiplicative subgroup $G$ of $A$. 

\begin{myeg}
Let $A$ be a commutative ring and $G$ be a subgroup of the group of multiplicative units of $A$. Then $G$ naturally acts on $A$ by multiplication. Let $A/G$ be the set of equivalence classes. We let $[x]$ be the equivalence class of $x \in A$. Then, $A/G$ becomes a hyperring with the following hyperaddition and multiplication: for any $[a], [b] \in A/G$, 
\[
[a]\cdot[b]=[ab], \quad [a]+[b]=\{[c] \mid c=g_1a+g_2b\textrm{ for some } g_1,g_2 \in G\}.
\]
A hyperring of the form $A/G$ is said to be a \emph{quotient hyperring}. 
\end{myeg}





\begin{mydef}\label{definition: submodule}
Let $H$ be a hyperring and $M$ be an $H$-module. By a submodule $N$ of $M$, we mean a subset of $M$ which is itself an $H$-module with the induced hyperaddition and scalar multiplication.
\end{mydef}

\begin{rmk}
We remark that in Definition \ref{definition: submodule}, for any $a,b \in N$ the hyperaddition $a+_Nb$ taken in $N$ is a subset of the hyperaddition $a+_Mb$ taken in $M$, and they do not have to be same. For example, the tropical hyperfield $\mathbb{T}$ is an $\mathbb{K}$-module, where $\mathbb{K}$ is the Krasner hyperfield. In this case $\mathbb{K}$ is a submodule of $\mathbb{T}$, but, $1+_{\mathbb{K}}1=\{0,1\}$ whereas $1+_{\mathbb{T}}1=[-\infty,1]$. 
\end{rmk}


\section{The category of modules over a semiring as a proto-exact category} \label{section: proto-exact for semiring}

In this section, we prove that the category of modules over a semiring is a proto-exact category. We define admissible monomorphisms (resp.~admissible epimorphisms) to be equalizers (resp.~coequalizers) in certain forms. Then, we give an equivalent description of admissible monomorphisms and admissible epimorphisms in terms of saturated submodules. 

Let $M$ be a module over a semiring $R$. Recall that a submodule $N$ of $M$ is said to be saturated if $x+y \in N$ and $y \in N$ implies that $x \in N$ for any $x, y \in M$. It is clear that the intersection of a family of saturated submodules of $M$ is again saturated. In particular, for each subset $S$ of $M$, there exists a smallest saturated submodule of $M$ containing $S$. We call this the saturation closure (or simply the saturation) of $S$, denoted by $\overline{S}$ or $\angles{S}$. The following is proved in \cite{jun2020lattices} for ideals, but the same proof works for submodules. We include the proof for completeness. 

\begin{pro}\label{proposition: saturation closure}
Let $R$ be a semiring. Let $M$ be an $R$-module and $N$ be a submodule of $M$. Then the smallest saturated submodule of $M$ containing $N$ is $\{x \in M \mid x+a=b \textrm{ for some }a,b \in N\}$.
\end{pro}
\begin{proof}
Let $L=\{x \in M \mid x+a=b \textrm{ for some }a,b \in N\}$. It is clear that $L$ is a submodule of $M$ containing $N$. If $x+y, y \in L$, then we have $x+y+a=b$ and $y+a'=b'$ for some $a,a',b,b' \in N$. Then, we have $x+a+(y+a')=x+a+b'$ and $(x+y+a)+a'=b+a'$. It follows that $x+a+b'=b+a'$, showing that $x \in L$. In particular, $L$ is saturated. Finally it is clear that any saturated submodule of $M$ containing $N$ should contain $L$.   
\end{proof}

When $R$ is an idempotent semiring, one has the following equivalent description of saturated submodules. Recall that any module $M$ over an idempotent semiring $R$ is equipped with a canonical partial order as follows: for any $x,y \in M$
\[
x \leq y \iff x+y=y.
\]

\begin{pro}(c.f.~\cite{jun2020lattices})\label{proposition: saturated equiv}
Let $R$ be an idempotent semiring. Let $M$ be an $R$-module and $N$ be a submodule of $M$. Then $N$ is saturated if and only if for all $x \in N$ and $y\leq x$ it follows that $y \in N$. 
\end{pro}
\begin{proof}
Suppose that $N$ is saturated. If $x \in N$ and $y\leq x$, then we have $x+y=x \in N$. Now $y \in N$ since $N$ is saturated. Conversely, suppose $N$ satisfies the given condition. Let $x \in N$ and $x+y \in N$. Then, we have $(x+y)+y=x+y \in N$ and hence $y \leq (x+y)$, implying that $y \in N$ from the given condition.
\end{proof}

\begin{lem}\label{lemma: saturation inverse}
Let $R$ be a semiring and $f:M \to N$ be a morphism of $R$-modules. If $L$ is a saturated submodule of $N$, then $f^{-1}(L)$ is a saturated submodule of $M$
\end{lem}
\begin{proof}
Let $K=f^{-1}(L)$. It is clear that $K$ is a submodule of $M$. If $x+y,y \in K$, then we have $f(x+y)=f(x)+f(y),f(y) \in L$. Since $L$ is saturated, we have $f(x) \in L$, in particular $x \in K$, showing that $K$ is saturated. 
\end{proof}

\begin{rmk}\label{remark: saturated quotient}
Let $M$ be an $R$-module and $N$ be a submodule of $M$. Then $M/N$ is isomorphic to $M/\angles{N}$, where $\angles{N}$ is the saturation closure of $N$. To see this, we only have to show that $N$ and $\angles{N}$ define the same equivalence relation. Let $\sim$ (resp.~$\equiv$) be the equivalence relation defined by $\angles{N}$ (resp.~$N$). Since $N\subseteq \angles{N}$, we only have to prove that if $x \sim y$ for $\angles{N}$, then $x \equiv y$ for $N$. Now, if $x \sim y$ for $\angles{N}$, then $x+a=y+b$ for some $a,b \in \angles{N}$. It follows from Proposition \ref{proposition: saturation closure} that $a+n_1=n_2$ and $b+n_1'=n_2'$ for some $n_1,n_1',n_2,n_2' \in N$. By combining these, we have
\[
x+a+n_1+n_1'=x+n_2+n_1'=y+b+n_1+n_1'=y+n_1+n_2',
\] 
showing that $x\equiv y$ for $N$. 
\end{rmk}

\begin{mydef}
Let $R$ be a semiring and $f:M \to N$ be a morphism of $R$-modules. 
\begin{enumerate}
	\item 
$f$ is said to be a \emph{normal monomorphism} if $f$ is the equalizer of some morphism and the zero map. 
\item 
$f$ is said to be a \emph{normal epimorphism} if $f$ is the coequalizer of some morphism and the zero map. 	
\end{enumerate}
\end{mydef}

\begin{rmk}\label{remark: inj surj}
Let $R$ be a semiring and $f:M \to N$ be a morphism of $R$-modules. 
\begin{enumerate}
	\item 
Since an equalizer in the category of modules over a semiring is just a set-theoretic equalizer, if $f$ is a normal monomorphism, then $f$ is injective. In particular, $f$ is a monomorphism. 
	\item 
One can easily observe that the following is a coequalizer diagram in the category of modules over $R$:
\begin{equation}
	\begin{tikzcd}
		L\arrow[r, shift left,"g"] \arrow[r,swap, shift right,"0"]
		& M \arrow[r, "\pi"] & M/g(L),
	\end{tikzcd}
\end{equation}
\end{enumerate}
where $\pi$ is a canonical projection. In fact, if $h:M \to K$ is an $R$-module morphism such that $hg=0$. Then, we can define a map $u:M/g(L) \to K$ sending $[m]$ to $h(m)$, where $[m]$ is the equivalence class of $m \in M$ in $M/g(L)$. Note that $u$ is well-defined; if $[m]=[n]$, then $m+x=n+y$ for some $x, y \in g(L)$. Hence, $h(m+x)=h(m)+h(x)=h(m)$ and similarly $h(n+y)=h(n)$. It is also clear that $u$ is unique once it exists. This shows that any normal epimorphism is surjective, in particular, it is an epimorphism. 
\end{rmk}

The following provides an equivalent description for normal monomorphisms and normal epimorphisms. 

\begin{lem}\label{lemma: normal modules iff}
Let $R$ be a semiring and $f:M \to N$ be a morphism of $R$-modules. 
\begin{enumerate}
	\item
$f$ is a normal monomorphism if and only if $f$ is injective and $f(M)$ is isomorphic to a saturated submodule of $N$.  
\item
$f$ is a normal epimorphism if and only if $N$ is isomorphic to the quotient of $M$ by a saturated submodule and $f$ is a natural projection.
\end{enumerate}
\end{lem}
\begin{proof}
$(1):$ Suppose that $f$ is a normal monomorphism. From Remark \ref{remark: inj surj}, we may assume that 
\[
M=\{n \in N \mid g(n)=0\}
\]
for some $g:N \to L$ and $f:M \to N$ is an inclusion. In particular, $M=g^{-1}(0)$ which is clearly saturated from Lemma \ref{lemma: saturation inverse} since $\{0\}$ is a saturated submodule. It follows that $f(M)$ is a saturated submodule of $N$. Conversely, we may assume that $M$ is a saturated submodule of $N$. Consider the natural projection $\pi:N \to N/M$. One can easily see that $f:M\to N$ is the equalizer of $\pi$ and the zero map since the equalizer should be the smallest saturated submodule of $N$ containing $M$ in this case. 

$(2):$ Suppose that $f$ is a normal epimorphism. It follows from Remarks \ref{remark: saturated quotient} and \ref{remark: inj surj} that $N$ is isomorphic to the quotient of $M$ by a saturated submodule and $f$ is a natural projection. The converse is clear. 
\end{proof}

\begin{mydef}\label{definition: proto semiring module}
Let $\Mod_R$ be the category of modules over a semiring $R$. Let $\mathfrak{M}$ be the class of normal monomorphisms in $\Mod_R$ and $\mathfrak{E}$ the class of normal epimorphisms in $\Mod_R$.
\end{mydef}

In the following, we prove that $(\Mod_R,\mathfrak{M},\mathfrak{E})$ is a proto-exact category. We will use the terms normal monomorphisms (resp.~normal epimorphisms) and admissible monomorphisms (resp.~admissible epimorphisms) interchangeably. We first need the following version of the third isomorphism theorem for modules over a semiring.\footnote{Since we cannot locate the proof of the statement, we include our proof here.}

\begin{lem}\label{lemma: third iso theorem}
Let $R$ be a semiring, $M$ an $R$-module, and $N$ a saturated submodule of $M$. Let $\pi:M \to M/N$ be a natural projection and $K$ be a saturated submodule of $M/N$. Then we have
\[
M/\pi^{-1}(K) \simeq (M/N)/K.
\]
\end{lem}
\begin{proof}
It is enough to show that the following natural map is well-defined and injective:
\[
\varphi: M/\pi^{-1}(K) \to (M/N)/K, \quad \overline{m} \mapsto [\widetilde{m}],
\]	
where $\overline{m}$ (resp.~$\widetilde{m}$) is the equivalence class of $m \in M$ in $M/\pi^{-1}(K)$ (resp.~in $M/N$) and $[\widetilde{m}]$ is the equivalence class of $\widetilde{m}$ in $(M/N)/K$. 

We first show that $\varphi$ is well-defined. For $m_1,m_2 \in M$, suppose that $\overline{m_1}=\overline{m_2}$. Then we have
\begin{equation}
m_1+c_1 = m_2 +c_2 \textrm{ for some } c_1,c_2 \in \pi^{-1}(K).
\end{equation}
It follows that
\begin{equation}\label{eq: third iso}
\pi(m_1)+\pi(c_1) = \pi(m_2) + \pi(c_2).
\end{equation}
Since $\pi(c_1),\pi(c_2) \in K$, \eqref{eq: third iso} implies that $\pi(m_1)=\widetilde{m_1}$ and $\pi(m_2)=\widetilde{m_2}$ are equivalent under $K$. In particular, $[\widetilde{m_1}]=[\widetilde{m_2}]$, showing that $\varphi$ is well-defined. 

Next, we show that $\varphi$ is injective. Suppose that $\varphi(\overline{m_1})=\varphi(\overline{m_2})$. So, $m_1$ and $m_2$ induce the same element of $(M/N)/K$, or there exist $k_1,k_2 \in K$ such that
\begin{equation}
\pi(m_1)+k_1=\pi(m_2)+k_2.
\end{equation} 
By surjectivity of $\pi$, we may choose $d_1$ and $d_2$ which map to $k_1$ and $k_2$ under $\pi$. This implies that
\begin{equation}
\pi(m_1) + \pi(d_1) = \pi(m_2) + \pi(d_2),
\end{equation}
in particular, $\overline{m_1+d_1}=\overline{m_2+d_2}$. Since $d_1,d_2 \in \pi^{-1}(K)$, we have that
\[
\overline{m_1}=\overline{m_1+d_1}=\overline{m_2+d_2}=\overline{m_2},
\]
showing that $\varphi$ is injective. 
\end{proof}		

\begin{lem}\label{lemma: composition lemma}
$\MM$ and $\EE$ in Definition \ref{definition: proto semiring module} are closed under composition and contain all isomorphisms.
\end{lem}
\begin{proof}
It is clear that $\MM$ and $\EE$ contain all isomorphisms. From the descriptions in Lemma \ref{lemma: normal modules iff}, it is clear that $\MM$ is closed under composition. So, we prove that $\EE$ is closed under composition. Let $g:L\to M$ and $f:M\to N$ be normal epimorphisms and $h:=fg$. From Lemma \ref{lemma: normal modules iff}, we may assume that $M= L/T$ for some saturated submodule $T$ of $L$ and $g:L \to L/T$ is a projection. Now, since $f:L/T \to N$ is a normal epimorphism, again from Lemma \ref{lemma: normal modules iff}, we may assume that $N=(L/T)/H$ for some saturated submodule $H$ of $(L/T)$ and $f$ is a projection. Let $\pi:L \to (L/T)/H$ be a projection. It follows from Lemmas \ref{lemma: saturation inverse} and \ref{lemma: third iso theorem} that $\pi^{-1}(H)$ is a saturated submodule of $L$ and $L/\pi^{-1}(H)=(L/T)/H$. Hence, we conclude that
\[
h=fg: L \to N=L/\pi^{-1}(H)
\]
is a projection. This shows that $\EE$ is closed under composition. 
\end{proof}

\begin{rmk}
One may be also tempted to use regular monomorphisms and regular epimorphisms to impose a proto-exact structure on $\Mod_R$. However, these classes do not behave well for our purpose. More explicitly, regular mono/epi-morphisms are not closed under composition in general. In particular, Lemma \ref{lemma: composition lemma} does not hold for regular morphisms. Note that even normal mono/epi-morphisms are not closed under composition in arbitrary pointed categories although they are closed under composition for $\Mod_R$.
\end{rmk}

\begin{lem}\label{lemma: bi-Cartesian lemma}
Let $R$ be a semiring and $M$ an $R$-module. Let $i \in \mathfrak{M}$, $\pi,\pi' \in \mathfrak{E}$, and $L$ a saturated submodule of $M$. Then, the following hold.
\begin{enumerate}
\item 
Projections $\pi:M \to M/L$ and $\pi':N\to N/L$ are normal epimorphisms.
	\item 
The induced map $i':M/L \to N/L$ is a normal monomorphism.
\item 
The following commutative diagram is a bi-Cartesian diagram in $\Mod_R$.
\begin{equation}\label{eq: Bicart}
	\begin{tikzcd}
		M
		\ar[r,"i",hook]
		\ar[d,"\pi",two heads,swap]
		&
		N
		\ar[d,"\pi'",two heads]
		\\
		M/L
		\ar[r,"i'",hook]
		&
		N/L
	\end{tikzcd}
\end{equation}
\end{enumerate}
\end{lem}
\begin{proof}
We first note that since $L$ is a saturated submodule of $M$ and $M$ is a saturated submodule of $L$ (as $i \in \MM$), $L$ is a saturated submodule of $N$. One can easily see from Lemma \ref{lemma: normal modules iff} that $\pi$ and $\pi'$ are normal epimorphisms.
	
Next, we claim that if $i:M \to N$ is a normal monomorphism, then the induced map $i':M/L \to N/L$ is a normal monomorphism. In fact, it is clear that $i'$ is injective. Hence, it is enough to prove that $i'(M/L)$ is saturated in $N/L$. We identify $M/L$ and $i'(M/L)$. Let $[x]$ and $[y]$ be the equivalence classes of $x,y \in N$ in $N/L$. Suppose that $[x]+[y]=[x+y] \in M/L$ and $[x] \in M/L$, that is, $[x+y]=[m]$ and $[x]=[m']$ for some $m,m' \in M$. It follows that
\[
x+y+\ell_1 = m+\ell_2, \quad x+\ell_3 = m'+\ell_4, \quad \ell_i \in L. 
\]
Hence we have
\[
(m+\ell_2)+\ell_3=x+\ell_3+(y+\ell_1) = m'+\ell_4+(y+\ell_1)=y+ (m'+\ell_1+\ell_4).
\]
Therefore $y+(m'+\ell_1+\ell_4) \in M$ and $(m'+\ell_1+\ell_4) \in M$. Since $i$ is a normal monomorphism, $M$ is saturated submodule of $N$, so we conclude that $y \in M$, showing that $[y] \in M/L$. In particular, \eqref{eq: Bicart} is well-defined and $i'\in \MM$. 

(co-Cartesian) Consider the following commutative diagram, where $K$ is an $R$-module and $\alpha,\beta$ are morphisms in $\Mod_R$:
	\begin{equation}\label{eq: pushout1}
	\begin{tikzcd}
	M \ar[hook]{r}{i} \ar[swap,two heads]{d}{\pi} & N \ar[two heads]{d}{\pi'}
	\ar[bend left]{ddr}{\beta} &
	\\
	M/L \ar[hook]{r}{i'} \ar[swap, bend right]{rrd}{\alpha} & N/L \ar[dotted]{dr}{\gamma} &
	\\
	& & K
	\end{tikzcd}
	\end{equation}
Since $L \subseteq M$, for any $\ell \in L$, we have
\[
\pi'i(\ell) = i'\pi(\ell) = 0_{N/L}. 
\]
Since $i'$ is injective (as it is a normal monomorphism), we have $\pi(\ell)=0_{M/L}$. It follows that
\begin{equation}\label{eq: lll}
	0_K=\alpha\pi(\ell)=\beta i(\ell)=\beta(\ell).
\end{equation}
Now, if $\pi'(x)=\pi'(y)$, then we have $x+\ell_1=y+\ell_2$ for some $\ell_i \in L$. Then, from \eqref{eq: lll}, we have
\[
\beta(x)=\beta(y).
\]
Therefore, the function $\gamma:N/L \to K$ sending $[x]$ to $\beta(x)$ is a well-defined morphism of $R$-modules. Furthermore, it is clear that with $\gamma$, \eqref{eq: pushout1} commutes, and if a morphism $\delta:N/L \to K$ makes \eqref{eq: pushout1} to commute, then $\delta=\gamma$. Hence $\eqref{eq: Bicart}$ is a co-Cartesian diagram. 
	
	(Cartesian) Consider the following diagram:
	\begin{equation}\label{eq: Catt}
	\begin{tikzcd}
	K \ar[bend left]{drr}{\beta} \ar[swap, bend right]{ddr}{\alpha} \ar[dotted]{dr}{\gamma} & &
	\\
	&M\ar[hook]{r}{i} \ar[swap,two heads]{d}{\pi} & N \ar[two heads]{d}{\pi'}
	\\
	& M/L \ar[hook]{r}{i'} & N/L
	\end{tikzcd}
	\end{equation}
where $K$ is an $R$-module and $\alpha,\beta$ are morphisms of $R$-modules. We claim that $\beta(K) \subseteq i(M)$. In fact, for any $x \in K$, we have $i'\alpha (x) = \pi'\beta(x)$. We may write $\alpha(x)=\pi(m)$ for some $m \in M$ since $\pi$ is surjective (as it is a normal epimorphism). Hence we have
	\[
	\pi'i(m)=i'\pi(m) = \pi'\beta(x).
	\]
	In particular, $\beta(x)+\ell_1= i(m) +\ell_2$ for some $\ell_i \in L$. But, since $L\subseteq M$ and $M$ is assumed to be saturated in $N$ (since $i$ is a normal monomorphism), we have that $\beta(x) \in i(M)$. Hence, the function $\gamma:K \to M$ sending $x$ to $\beta(x)$ is a well-defined morphism of $R$-modules, which makes the diagram \eqref{eq: Catt} to commute. It is clear that such $\gamma$ is unique. This proves that \eqref{eq: Bicart} is a Cartesian diagram. 
\end{proof}

\begin{lem}\label{lemma: pro4}
	Every diagram in $\Mod_R$
	\begin{equation}\label{eq: pro4}
	\begin{tikzcd}
	{} & N\ar["j'",d,two heads] \\
	M' \ar["i'",r,hook,swap] & N' 
	\end{tikzcd}
	\end{equation}
	with $i' \in \MM$ and $j' \in \EE$ can be completed to a
	bi-Cartesian square \eqref{eq: bi-cat} with $i \in \MM$ and
	$j \in \EE$.
\end{lem}
\begin{proof}
Let $M= (j')^{-1}(i'(M'))$. Since $i'(M')$ is a saturated submodule of $N'$, it follows from Lemma \ref{lemma: saturation inverse}, $M$ is a saturated submodule of $N$. In particular the inclusion $i:M \to N$ is a normal monomorphism. 
	
We also have a canonical surjective homomorphism $\pi:M \to M'$ sending any $m \in M$ to $((i')^{-1}j'i)(m)$. Note that $\pi$ is well-defined since $j'i(M) \subseteq i'(M')$ by definition. We claim that $\pi$ is a normal epimorphism. In fact, since $j':N \to N'$ is a normal epimorphism, we may assume that $j':N \to N/L$, where $L$ is a saturated submodule of $N$. Then, by definition of $M$, we have $L\subseteq M$. In particular, we may assume that
\[
\pi:M \to M'=M/L 
\]
and hence $\pi$ is a normal epimorphism since $L$ is a saturated submodule of $M$. Therefore, we can complete the diagram \eqref{eq: pro4} as follows:
	\begin{equation}\label{eq: Bicar}
	\begin{tikzcd}
	M
	\ar[r,"i",hook]
	\ar[d,"\pi",two heads,swap]
	&
	N
	\ar[d,"j'",two heads]
	\\
	M/L
	\ar[r,"i'",hook]
	&
	N/L
	\end{tikzcd}
	\end{equation}
Now it follows from Lemma \ref{lemma: bi-Cartesian lemma} that \eqref{eq: Bicar} is a bi-Cartesian diagram.
\end{proof}

\begin{lem}\label{lemma: eq5}
	Every diagram in $\Mod_R$
	\begin{equation}\label{eq: pro5}
	\begin{tikzcd}
	M \ar["j",d,two heads,swap] \ar["i",r,hook] & N\\
	M'  & {}
	\end{tikzcd}
	\end{equation}
	with $i \in \MM$ and $j \in \EE$ can be completed to a
	bi-Cartesian square \eqref{eq: bi-cat} with $i' \in \MM$ and
	$j' \in \EE$.
\end{lem}
\begin{proof}
Since $j$ is a normal epimorphism, we may assume that $j:M \to M'=M/L$, where $L$ is a saturated submodule of $M$. We let $j':N \to N/L$ be a natural projection map, which is clearly a normal epimorphism since $L$ is also a saturated submodule of $N$.\footnote{Since $M$ is a saturated submodule of $N$, if $x,y \in N$ and $x+y,y \in L$, then $x+y,y \in M$. So, this follows from the fact that $L$ is a saturated submodule of $M$.} As in the proof of Lemma \ref{lemma: bi-Cartesian lemma}, $i$ induces a natural map
\[
i': M/L \to N/L
\]
which is a normal monomorphism. In particular, we can complete the diagram \eqref{eq: pro5} as follows:
\begin{equation}\label{eq: propro}
	\begin{tikzcd}
	M
	\ar[r,"i",hook]
	\ar[d,"\pi",two heads,swap]
	&
	N
	\ar[d,"j'",two heads]
	\\
	M/L
	\ar[r,"i'",hook]
	&
	N/L
	\end{tikzcd}
\end{equation}
Now, it follows from Lemma \ref{lemma: bi-Cartesian lemma} that \eqref{eq: propro} is a bi-Cartesian diagram. 
\end{proof}

\begin{mythm}\label{theorem: main theorem for modules}
Let $R$ be a semiring. Then the triple $(\Mod_R,\MM,\EE)$ as in Definition \ref{definition: proto semiring module} defines a proto-exact category.
\end{mythm}
\begin{proof}
$\Mod_R$ is clearly pointed. The axioms $(2),(4)$, and $(5)$ in Definition \ref{definition: proto_exact} directly follow from Lemmas \ref{lemma: composition lemma}, \ref{lemma: pro4}, and \ref{lemma: eq5}. Finally, the axiom $(3)$ is a direct consequence of $(4)$ and $(5)$ along with the uniqueness of pullbacks and pushouts. 
\end{proof}

A finitely generated module $N$ over a semiring $R$ is said to be \emph{Noetherian} if any submodule of $N$ is also finitely generated. The following is obvious from Theorem \ref{theorem: main theorem for modules}. 

\begin{cor}
Let $R$ be a semiring such that any finitely generated module $N$ over $R$ is Noetherian.\footnote{For instance, this is true when $R=\mathbb{B}$.} Let $\Mod_R^{\emph{fg}}$ be a subcategory of $\Mod_R$ consisting of finitely generated $R$-modules with the same exact structure, then $\Mod_R^{\emph{fg}}$ is also a proto-exact category. 
\end{cor}

\begin{rmk}
When one has a proto-exact category $\mathcal{C}$, then one may study the Hall algebra $H_\mathcal{C}$ of $\mathcal{C}$ or develop $K$-theory of $\mathcal{C}$. When $\mathcal{C}$ is the category of matroids, these two applications were considered in \cite{eppolito2018proto}. In this paper, with these two applications in mind, we equip the category of modules over a semiring with a proto-exact structure. In \cite{jun2020vector}, some ingredients of $K$-theory for semirings were introduced without appealing to any proto-exact structure. 
\end{rmk}


\section{Algebraic lattices, $\mathbb{B}$-modules, and matroids} \label{section: main lattices} \label{section: lattice section}

In this section, we focus on the case when $R=\mathbb{B}$, the Boolean semifield. It is well-known that any finite $\mathbb{B}$-module is naturally equipped with a lattice structure.\footnote{To be precise, any $\mathbb{B}$-module defines a join semilattice.} See, for instance, \cite[\S 2]{les3}.

We prove that the category $\mathcal{L}$ of algebraic lattices is proto-exact, and explore certain connections to $\mathbb{B}$-modules and the subcategory of geometric lattices (i.e.~simple matroids).	

\subsection{Algebraic lattices}

In this subsection, we prove that the category of algebraic lattices is equipped with a proto-exact structure. We first recall some necessary definitions. 

\begin{mydef}
Let $L$ be a lattice.
\begin{enumerate}
	\item 
By a \emph{cover} of an element $x \in L$, we mean a family of elements $\{y_i\}_{i \in I}$ such that 
\[
x \leq \underset{i \in I}{\vee}~ y_i.
\]
\item 
For a complete lattice $L$, an element $x \in L$ is said to be \emph{compact} if every cover of $x$ has a finite subcover.
\item 
An \emph{algebraic lattice} is a complete lattice such that any element is a join of compact elements. 
\end{enumerate}
\end{mydef}

 Note that any finite lattice is algebraic, but we will use the term ``finite algebraic lattices'' to emphasize that we work with algebraic lattices in general rather than lattices. 

\begin{rmk}
Equivalently, one may define an algebraic lattice as a poset which is locally finitely presentable as a category.
\end{rmk}

\begin{mydef}\label{defnition: morphis of algebraic lattices}
Let $L_1$ and $L_2$ be algebraic lattices. A function $f:L_1 \to L_2$ is said to be a \emph{morphism of algebraic lattices} if the following hold:
\begin{enumerate}
\item 
$f$ preserve arbitrary joins: for any $\{y_i\}_{i \in I}  \subseteq L_1$ we have
\[
f(\underset{i \in I}{\vee}~ y_i)=\underset{i \in I}{\vee}~ f(y_i).
\]
\item 
If $x \in L_1$ is compact, then $f(x)$ is compact element in $L_2$.
\end{enumerate}
\end{mydef} 

\begin{rmk}
We note that in our definition of morphisms for algebraic lattices, we do not require them to preserve meets. In particular, a morphism of algebraic lattices as in Definition \ref{defnition: morphis of algebraic lattices} does not have to be a morphism of underlying lattices. 
\end{rmk}

\begin{mydef}\label{definition: normal maps for lattices}
Let $L_1$ and $L_2$ be algebraic lattices and $f:L_1 \to L_2$ be a morphism of algebraic lattices as in Definition \ref{defnition: morphis of algebraic lattices}.  
\begin{enumerate}
\item 
$f$ is said to be a \emph{normal monomorphism} if
\begin{enumerate}
	\item
 $f$ is an injection;
 \item 
 $f(L_1)$ is downward closed, i.e., $\forall~x \in f(L_1)$ and $\forall~y \in L_2$ if $y \leq x$ then $y \in f(L_1)$.
\end{enumerate}
\item 
$f$ is said to be a \emph{normal epimorphism} if there exists $x_0 \in L_1$ such that 
\begin{enumerate}
	\item
$L_2 \simeq \{y \in L_1\mid y \geq x_0\}$;
\item 
$f:L_1 \to L_2$ is given by joining  with $x_0$ under the above isomorphism. 
\end{enumerate}
\end{enumerate}
\end{mydef}

\begin{mydef}\label{definition: algebraic sublattice}
Let $L$ be an algebraic lattice. A subset $K \subseteq L$ is said to be an \emph{algebraic sub-lattice} if $K$ itself is an algebraic lattice and the inclusion $i:K \to L$ is a normal monomorphism. 
\end{mydef}

Let $L$ be an algebraic lattice and $K$ be an algebraic sub-lattice of $L$. Let $1_K$ be the unique maximal element of $K$. By the quotient lattice $L/K$, we mean the following algebraic sub-lattice of $L$:
\[
L/K:=\{x \in L \mid x\geq 1_K\}.
\]
The canonical map $f:L \to L/K$ sending $x$ to $(x\vee 1_K)$ is a normal epimorphism. 

\begin{mydef}\label{definition: algebraic lattice = proto}
Let $\mathcal{L}$ be the category of algebraic lattices. Let $\mathfrak{M}$ be the class of normal monomorphisms in $\mathcal{L}$ and $\mathfrak{E}$ the class of normal epimorphisms in $\mathcal{L}$.
\end{mydef}

\begin{lem}\label{lemma: composition lemma lattice}
$\MM$ and $\EE$ in Definition \ref{definition: normal maps for lattices} are closed under composition and contain all isomorphisms.
\end{lem}
\begin{proof}
It is clear that $\MM$ and $\EE$ contain all isomorphisms. 

To show that $\MM$ is closed under composition, suppose that $f_1:L_1 \to L_2$ and $f_2:L_2 \to L_3$ are normal monomorphisms and $g:=f_2f_1$. Clearly, $g$ is an inclusion. Suppose further that $x \in g(L_1)$, $y \in L_3$, and $y \leq x$. We write $x=f_2(f_1(z))$. Since $y \leq f_2(f_1(z))$ and $f_2$ is a normal monomorphism, we have that $y \in f_2(L_2)$. Let $y=f_2(t)$ for some $t \in L_2$. Then, we have $f_2(t)=y \leq x=f_2(f_1(z))$. Since $f_2$ is an inclusion, we have $t \leq f_1(z)$. Again, since $f_1(L_1)$ is downward closed, we have that $t \in f_1(L_1)$, and hence $y=f_2(t) \in g(L_1)$. 

To show that $\EE$ is closed under composition, suppose that $f_1:L_1 \to L_2$ and $f_2:L_2 \to L_3$ are normal epimorphisms. From the definition, we may assume that $f_1$ is the composition of the map $L_1 \to L_1'=\{x \in L_1 \mid x \geq x_1\}$ given by joining with some element $x_1$ and an isomorphism $u:L_1' \to L_2$. Similarly, $f_2$ is the composition of the map $L_2 \to L_2'$ given by joining with some element $x_2$ and an isomorphism $v: L_2' \to L_3$. Then we have
\[
f_2(f_1(x))=f_2 ( u(x_1 \vee x)) = v(x_2 \vee u(x_1 \vee x))=v ( u (u^{-1}(x_2) \vee x_1 \vee x)).
\]
Thus $f_2f_1$ is the composition of the join map of $u^{-1}(x_2)\vee x_1$ with the map $vu: u^{-1}(L_2) \to L_3$.\footnote{Of course, we have to restrict the domain of $u$ in order for $vu$ to make sense.} $vu$ is an isomorphism because both $v$ and the map $u:u^{-1}(L_2) \to L_2$ are isomorphisms. Now, it follows from the definition that $f_2f_1$ is a normal epimorphism. 
\end{proof}

\begin{lem}\label{lemma: common lemma}
Let $M$ be an algebraic lattice and $L$ an algebraic sub-lattice of $M$. Then, any compact element in $M/L$ is the join of a compact element of $M$ with $1_L$.
\end{lem}
\begin{proof}
Let $x$ be a compact element of $M/L$, and by viewing $x$ as an element of $M$, we can write it as the join of compact elements $\{y_\alpha\}$ of $M$ (where $\alpha$ ranges over some indexing set).  Then $x$ is also the join of elements of the form $y_\alpha \vee 1_L$, and since these lie in $M/L$ we can pick a finite subcover.  This allows one to write $x = y_1 \vee \dots \vee y_n \vee 1_L$, showing that $y_1 \vee \dots\vee y_n$ is compact. 		
\end{proof}

Now, we prove that $(\mathcal{L},\MM,\EE)$ is a proto-exact category. As in the case of semirings, we will interchangeably use the terms normal monomorphisms (resp.~normal epimorphisms) and admissible monomorphisms (resp.~admissible epimorphisms). We will make use of the following lemma whose proof is similar to Lemma \ref{lemma: bi-Cartesian lemma}. 

\begin{lem}\label{lemma: bi-Cartesian key lemma}
Let $M$ and $N$ be algebraic lattices and $L$ an algebraic sub-lattice of $M$. Let $i \in \mathfrak{M}$, $\pi,\pi' \in \mathfrak{E}$. Then, the induced map $i':M/L \to N/L$ is a normal monomorphism and the following commutative diagram is a bi-Cartesian diagram in $\mathcal{L}$.
	\begin{equation}\label{eq: bicar diagram lattice}
	\begin{tikzcd}
	M
	\ar[r,"i",hook]
	\ar[d,"\pi",two heads,swap]
	&
	N
	\ar[d,"\pi'",two heads]
	\\
	M/L
	\ar[r,"i'",hook]
	&
	N/L
	\end{tikzcd}
	\end{equation}
\end{lem}
\begin{proof}
One can easily observe that if $i:M \to N$ is a normal monomorphism, then the induced map $i':M/L \to N/L$ is also a normal monomorphism. In fact, the only nontrivial part is to show that $i'$ preserves compact elements which directly follows from Lemma \ref{lemma: common lemma}.

(co-Cartesian) Consider the following commutative diagram, where $K$ is an algebraic lattice and $\alpha,\beta$ are morphisms in $\mathcal{L}$:
	\begin{equation}\label{eq: pushout lattice}
	\begin{tikzcd}
	M \ar[hook]{r}{i} \ar[swap,two heads]{d}{\pi} & N \ar[two heads]{d}{\pi'}
	\ar[bend left]{ddr}{\beta} &
	\\
	M/L \ar[hook]{r}{i'} \ar[swap, bend right]{rrd}{\alpha} & N/L \ar[dotted]{dr}{\gamma} &
	\\
	& & K
	\end{tikzcd}
	\end{equation}
From the definition of $N/L$, we can define the function $\gamma:N/L \to K$ such that $\gamma(x)=\beta(x)$. We first claim that $\gamma$ is a morphism of algebraic lattices as in Definition \ref{defnition: morphis of algebraic lattices}. If $\{y_i\}\subseteq N/L$, we may assume that $\{y_i\}=\{x_i\vee 1_L\}$ for some $x_i \in N$. In particular, for a nonempty join, we have
\[
\gamma(\vee y_i)=\gamma (\vee(x_i \vee 1_L))=\beta (\vee(x_i \vee 1_L)) = \vee (\beta (x_i\vee 1_L))=\vee \beta(y_i)=\vee \gamma(y_i),
\]
showing that $\gamma$ preserves nonempty joins. For the empty join, since $\beta i =\alpha \pi$ and $\pi(1_L)=0_{M/L}$, we have that $\beta(1_L)=0$. It follows from the definition that $\gamma(1_L)=\beta(1_L)=0$. In particular, since the empty join in $N/L$ is $1_L$, $\gamma$ preserves the empty join. Hence, $\gamma$ preserves arbitrary joins. Furthermore, since $\beta$ is a morphism of algebraic lattices, $\beta$ preserves compact elements. It follows from Lemma \ref{lemma: common lemma} that $\gamma$ preserves compact elements. 

Next, one may observe that $\gamma\pi'=\beta$ since $\beta(1_L)$ is the minimal element in $N/L$ and $\gamma$ is a morphism of algebraic lattices. One can also easily check that $\alpha=\gamma i'$. This shows the commutativity of \eqref{eq: pushout lattice} with $\gamma$. 

Finally, if a morphism $\delta:N/L \to K$ makes \eqref{eq: pushout lattice} to commutes, then $\delta=\gamma$. Hence \eqref{eq: bicar diagram lattice} is a co-Cartesian diagram. 
	
(Cartesian) Consider the following diagram:
\begin{equation}\label{eq: Cat lattice}
	\begin{tikzcd}
	K \ar[bend left]{drr}{\beta} \ar[swap, bend right]{ddr}{\alpha} \ar[dotted]{dr}{\gamma} & &
	\\
	&M\ar[hook]{r}{i} \ar[swap,two heads]{d}{\pi} & N \ar[two heads]{d}{\pi'}
	\\
	& M/L \ar[hook]{r}{i'} & N/L
	\end{tikzcd}
	\end{equation}
	where $K$ is an algebraic lattice and $\alpha,\beta$ are morphisms in $\mathcal{L}$. We claim that $\beta(K) \subseteq i(M)$. In fact, for any $x \in K$, we have $i'\alpha (x) = \pi'\beta(x)$. We may write $\alpha(x)=\pi(m)$ for some $m \in M$ since $\pi$ is surjective (as it is a normal epimorphism). Hence we have
	\[
	\pi'i(m)=i'\pi(m) = \pi'\beta(x).
	\]
In other words, we have
\[
\beta(x) \vee 1_L = i(m) \vee 1_L.
\]	
Now, since $L$ is an algebraic sub-lattice of $M$, we have that $i(m) \vee 1_L \in i(M)$, and hence $\beta(x) \vee 1_L \in i(M)$. However, since $i$ is a normal monomorphism, $i(M)$ is downward closed, showing that $\beta(x) \in i(M)$. Hence, the function $\gamma:K \to M$ sending $x$ to $\beta(x)$ is a well-defined morphism in $\mathcal{L}$, which makes the diagram \eqref{eq: Cat lattice} to commute. It is clear that such $\gamma$ is unique. This proves that \eqref{eq: bicar diagram lattice} is a Cartesian diagram. 
\end{proof}

\begin{lem}\label{lemma: algebraic sub-lattice}
Let $M$ be an algebraic lattice and $K$ be a subset of $M$ which is closed under arbitrary joins and also is downward closed. Then $K$ is an algebraic sub-lattice.
\end{lem}
\begin{proof}
To show that $K$ is a complete lattice,  it is well-known that one only needs to show that $K$ contains arbitrary joints. In particular, it follows from our assumption that $K$ is a complete lattice. Furthermore, every element of $K$ is a join of compact elements of $M$; since $K$ is downward closed, these are elements of $K$ as well (and are easily seen to be compact inside $K$). So $K$ is an algebraic lattice.

Next, we prove that the inclusion is a normal monomorphism. If $x \in K$ is compact and the $S=\{y_i\}_{i \in I}$ is a cover of $x$ in $M$, then $S'=\{y_i \wedge 1_K\}$ is a cover of $x$ in $K$. By picking a finite subcover of $S'$, we obtain a finite subcover of $S$ for $x$ in $M$. It follows that $x$ is compact in $M$, so the inclusion preserves compactness.  Since it clearly preserves joins, it is a morphism of algebraic lattices. Finally, since $K$ is downward closed, it is an algebraic sub-lattice.
\end{proof}

\begin{lem}\label{lemma: sub-lattice lemma}
Let $f:M \to N$ be a morphism in $\mathcal{L}$. If $L$ is an algebraic sub-lattice of $N$ (as in Definition \ref{definition: algebraic sublattice}), then $f^{-1}(L)$ is an algebraic sub-lattice of $M$.
\end{lem}
\begin{proof}
From Lemma \ref{lemma: algebraic sub-lattice}, we only have to prove that the inclusion $i:f^{-1}(L) \to M$ is closed under joins and is downward closed. 

All joins will take place in $M,N,$ or $L$ (there is no ambiguity since joins in $L$ and $N$ are the same). If $\{x_i\}_{i \in I} \subseteq f^{-1}(L)$, then we have 
\[
f(\vee x_i)=\vee f(x_i) \in L.
\]
Hence $f^{-1}(L)$ is closed under joins.

To check that $f^{-1}(L)$ is downward closed, suppose that $y \in f^{-1}(L)$ and $x\leq y$.  Then $f(x)\leq f(y)$, and hence $f(x) \in L$, showing that $f^{-1}(L)$ is downward closed. 
\end{proof}

\begin{lem}\label{lemma: pro4 lattice}
Every diagram in $\mathcal{L}$
\begin{equation}\label{eq: pro4 lattice}
\begin{tikzcd}
{} & N\ar["j'",d,two heads] \\
M' \ar["i'",r,hook,swap] & N' 
\end{tikzcd}
\end{equation}
with $i' \in \MM$ and $j' \in \EE$ can be completed to a
bi-Cartesian square \eqref{eq: bi-cat} with $i \in \MM$ and
$j \in \EE$.
\end{lem}
\begin{proof}
Let $M= (j')^{-1}(i'(M'))$. Since $i'(M')$ is an algebraic sub-lattice of $N'$, it follows from Lemma \ref{lemma: sub-lattice lemma}, $M$ is an algebraic sub-lattice of $N$. In particular the inclusion $i:M \to N$ is a normal monomorphism. 

Next, consider the function $\pi:M \to M'$ sending any $m \in M$ to $((i')^{-1}j'i)(m)$. Note that $\pi$ is well-defined since $j'i(M) \subseteq i'(M')$ by definition. 
We claim that $\pi$ is a normal epimorphism. In fact, since $j':N \to N'$ is a normal epimorphism, we may assume that $j':N \to N/L$, where $L$ is an algebraic sub-lattice of $N$. Then, by definition of $M$, we have $L\subseteq M$ (we identify $i(M)$ and $M$). In particular, we may assume that
	\[
	\pi:M \to M'=M/L 
	\]
	and hence $\pi$ is a normal epimorphism since $L$ is an algebraic sub-lattice of $M$. Therefore, we can complete the diagram \eqref{eq: pro4 lattice} as follows:
	\begin{equation}\label{eq: Bicar lattice}
	\begin{tikzcd}
	M
	\ar[r,"i",hook]
	\ar[d,"\pi",two heads,swap]
	&
	N
	\ar[d,"j'",two heads]
	\\
	M/L
	\ar[r,"i'",hook]
	&
	N/L
	\end{tikzcd}
	\end{equation}
	Now it follows from Lemma \ref{lemma: bi-Cartesian key lemma} that \eqref{eq: Bicar lattice} is a bi-Cartesian diagram.
\end{proof}

\begin{lem}\label{lemma: eq5 lattice}
	Every diagram in $\mathcal{L}$
	\begin{equation}\label{eq: pro5 lattice}
	\begin{tikzcd}
	M \ar["j",d,two heads,swap] \ar["i",r,hook] & N\\
	M'  & {}
	\end{tikzcd}
	\end{equation}
	with $i \in \MM$ and $j \in \EE$ can be completed to a
	bi-Cartesian square \eqref{eq: bi-cat} with $i' \in \MM$ and
	$j' \in \EE$.
\end{lem}
\begin{proof}
	Since $j$ is a normal epimorphism, we may assume that $j:M \to M'=M/L$, where $L$ is an algebraic sub-lattice of $M$. We let $j':N \to N/L$ be the map sending $x$ to $(x \vee 1_L)$, which is a normal epimorphism by definition. One can easily see that $i$ induces a natural map
	\[
	i': M/L \to N/L, \quad x \mapsto i(x), 
	\]
	which is a normal monomorphism. In particular, we can complete the diagram \eqref{eq: pro5 lattice} as follows:
	\begin{equation}\label{eq: propro lattice}
	\begin{tikzcd}
	M
	\ar[r,"i",hook]
	\ar[d,"\pi",two heads,swap]
	&
	N
	\ar[d,"j'",two heads]
	\\
	M/L
	\ar[r,"i'",hook]
	&
	N/L
	\end{tikzcd}
	\end{equation}
	Now, it follows from Lemma \ref{lemma: bi-Cartesian key lemma} that \eqref{eq: propro lattice} is a bi-Cartesian diagram. 
\end{proof}

\begin{mythm}\label{theorem: main theorem for lattices}
The triple $(\mathcal{L},\MM,\EE)$ as in Definition \ref{definition: algebraic lattice = proto} defines a proto-exact category.
\end{mythm}
\begin{proof}
	$\mathcal{L}$ is clearly pointed. The axioms $(2),(4)$, and $(5)$ in Definition \ref{definition: proto_exact} directly follow from Lemmas \ref{lemma: composition lemma lattice}, \ref{lemma: pro4 lattice}, and \ref{lemma: eq5 lattice}. Finally, the axiom $(3)$ is a direct consequence of $(4)$ and $(5)$ along with the uniqueness of pullbacks and pushouts. 
\end{proof}

\subsection{Connection to $\mathbb{B}$-modules and geometric lattices}

In this subsection, we study relations between the category of $\mathbb{B}$-modules and the category of algebraic lattices $\mathcal{L}$. We refer the reader to \cite[\S 3]{jun2020lattices} more on this line of ideas.  

Let $M$ be a $\mathbb{B}$-module and $S(M)$ be the set of all saturated submodules of $M$. It is well-known that $S(M)$ is an algebraic lattice. Let $S(M)^c$ be an algebraic lattice of $S(M)$ consisting of compact elements of $S(M)$. The following is also well-known. 

\begin{mythm}\label{theorem: lattice B-mdoule thm}
With the same notation as above, we have the following.
\begin{enumerate}
	\item 
$M\simeq S(M)^c$ (as $\mathbb{B}$-modules).
	\item
Let $L$ be an algebraic lattice. Then $L^c$ is equipped with a $\mathbb{B}$-module structure with the order induced from $L$. Moreover, $L$ is isomorphic to $S(L^c)$ (as lattices), where we consider $L^c$ as a $\mathbb{B}$-module.	
\end{enumerate}
\end{mythm}

The following lemma will be handy. 

\begin{lem}\cite{jun2020lattices}\label{lemma: saturation closure}
Let $M$ be a $\mathbb{B}$-module and $N$ be a submodule of $M$. The saturated submodule of $M$ generated by $N$ is given as follows:
\[
\{x \in M \mid \exists y \in N\textrm{ such that } x+y \in N\}.
\]
\end{lem}

For the following lemma, we let $\angles{N}$ be the saturated submodule generated by $N$.

\begin{lem}\label{lemma: vee lemma}
Let $f:M \to N$ be a morphism of $\mathbb{B}$-modules. Let $\{M_i\}$ be a family of submodules of $M$, $T$ be the submodule of $M$ generated by $\cup M_i$, and $L=\angles{T}$. Then we have
\[
\angles{f(L)} = \angles{\cup f(M_i)}.
\] 
\end{lem}
\begin{proof}
Let $K=\angles{\cup f(M_i)}$. It is clear that $K \subseteq \angles{f(L)}$. Conversely, suppose that $y \in \angles{f(L)}$. It follows from Lemma \ref{lemma: saturation closure} that there exists $x \in f(L)$ such that
\[
y+x \in f(L).
\] 
Take $z\in L$ such that $f(z)=y+x$. Again, from Lemma \ref{lemma: saturation closure}, there exists $q \in T$ such that $z+q \in T$. In particular, we can find $x_i,y_i\in \cup M_i$ such that
\[
z+q=\sum x_i, \quad q=\sum y_i.
\]
In particular, we have
\begin{equation}\label{eq: s1}
y+x+f(q)=f(z)+f(q)=f(z+q)=\sum f(x_i), \quad f(q)=\sum f(y_i). 
\end{equation}
On the other hand, since $x \in f(L)$, we may write $x=f(w)$ for some $w \in L$. So, there exists $q' \in T$ such that $w+q' \in T$. In particular, we can find $a_i,b_i \in \cup M_i$ such that 
\[
w+q'=\sum a_i, \quad q'=\sum b_i. 
\]
Hence we have
\begin{equation}\label{eq: s2}
f(w+q')=f(w)+f(q')=x+f(q')=\sum f(a_i)=x+\sum f(b_i).
\end{equation}
By combining \eqref{eq: s1} and \eqref{eq: s2}, we have
\[
y+x+f(q)+f(q')=y+\sum f(a_i) +f(q)=y+ \sum f(a_i)+\sum f(y_i) = \sum f(x_i) +\sum f(b_i). 
\]
Since the addition is idempotent, we further obtain:
\[
y+ ( \sum f(a_i) + \sum f(b_i) + \sum f(x_i) + \sum f(y_i))=( \sum f(a_i) + \sum f(b_i) + \sum f(x_i) + \sum f(y_i)),
\]
showing that $y \in K$. 
\end{proof}

\begin{lem}\label{lemma: compact lem1}
	Let $f:M \to N$ be a morphism of $\mathbb{B}$-modules. Let $K=\{m \in M \mid m \leq b\}$ for some $b$. Then $K$ is a saturated submodule of $M$ and $\angles{f(K)}=\{n \in N \mid n \leq f(b)\}$.
\end{lem}
\begin{proof}
From Proposition \ref{proposition: saturated equiv}, one may easily observe that $K$ is a saturated submodule of $M$. Let $L=\{n \in N \mid n \leq f(b)\}$. It is clear that $f(K) \subseteq L$ and $L$ is a saturated submodule of $N$ (again from Proposition \ref{proposition: saturated equiv}). So, it is enough to show that any saturated submodule $T$ of $N$ containing $f(K)$ contains $L$ as well. However, as $b \in K$, we have $f(b) \in f(K) \subseteq T$. It follows that if $n \leq f(b)$ then $n \in T$ since $T$ is saturated. Hence $L \subseteq T$.
\end{proof}

\begin{lem}\label{lemma: compact lemma}
Let $M$ be a $\mathbb{B}$-module and $N$ be a saturated submodule of $M$. Then $N$ is compact (as an element of the algebraic lattice $S(M)$) if and only if $N=\{a \in M \mid a\leq b\}$ for some $b \in M$.
\end{lem}
\begin{proof}
First, suppose that $N$ is compact. Let $S=\{\angles{n}\}_{n \in N}$ be a cover of $N$ in $S(M)$. Since $N$ is compact, there exists a finite subcover, say $\{\angles{n_1},\dots,\angles{n_r}\}$. Let $b=\sum_{i=1}^r n_i$. Then for any $a \in N$, we have $a \leq b$. Now, from Proposition \ref{proposition: saturated equiv}, $N=\{a \in M \mid a\leq b\}$.

Conversely, suppose that $N=\{a \in M \mid a \leq b\}$ for some $b \in M$. Let $S=\{K_i\}_{i \in I}$ be a cover of $N$. Since $S$ is a cover, after relabeling if needed, we may assume that there exist $k_i \in K_i$ (for $i=1,\dots,t$) such that $b \leq \sum_{i=1}^tk_i$. It follows that $\{K_1,\dots,K_t\}$ is a finite subcover, showing that $N$ is compact. 
\end{proof}

\begin{pro}
Theorem \ref{theorem: lattice B-mdoule thm} is functorial and hence one obtains the following functor:
\[
S:\emph{Mod}_\mathbb{B} \to \mathcal{L}, \quad M\mapsto S(M).
\]
\end{pro}
\begin{proof}
Let $M,N$ be $\mathbb{B}$-modules. For each $f \in \Hom(M,N)$, consider the function $S(f):S(M) \to S(N)$ sending any saturated submodule $L$ of $M$ to the saturated submodule of $N$ generated by $f(L)$. It follows from Lemma \ref{lemma: vee lemma} that $S(f)$ preserves arbitrary joins. Next, suppose that $y \in S(M)$ is compact. We claim that $z=S(f)(y)$ is compact. In fact, from Lemma \ref{lemma: compact lemma} a saturated submodule $y$ is compact if and only if it has the form $y=\{a \mid a \leq b\} \subseteq M$ for some $b \in M$. Furthermore, from Lemma \ref{lemma: compact lem1}, we have that $z=\{d \mid d \leq f(b)\}$, showing that $z$ is compact. This proves that $S(f)$ is a morphism of algebraic lattices. It remains to check that $S(fg)=S(f)S(g)$ for morphisms $f:M \to N$ and $g:L \to M$ in $\textrm{Mod}_\mathbb{B}$. But, this is clear from the definition.
\end{proof}

\begin{lem}\label{lemma: admissible lemma}
Let $f:M\to N$ be a morphism of $\mathbb{B}$-modules. 
\begin{enumerate}
	\item
If $f$ is a normal monomorphism in $\emph{Mod}_\mathbb{B}$, then $S(f)$ is a normal monomorphism in $\mathcal{L}$. 
\item 
If $f$ is a normal epimorphism in $\emph{Mod}_\mathbb{B}$, then $S(f)$ is a normal epimorphism in $\mathcal{L}$. 
\end{enumerate}
\end{lem}
\begin{proof}
$(1):$ From Lemma \ref{lemma: normal modules iff}, $f$ is a normal monomorphism in $\textrm{Mod}_\mathbb{B}$ if and only if $f(M)$ is a saturated submodule of $N$. In particular, this implies that $S(f):S(M) \to S(N)$ is an inclusion of algebraic lattices which is downward closed, showing that $S(f)$ is a normal monomorphism in $\mathcal{L}$. 

$(2):$ We first claim that if $L$ is a saturated submodule of a $\mathbb{B}$-module $M$, then we have
\[
S(M/L)=S(M)/S(L).
\]
In fact, the algebraic lattice $S(M/L)$ consists of saturated submodules of $M$ containing $L$. This is precisely the definition of $S(M)/S(L)$ as an algebraic lattice. Now, from Lemma \ref{lemma: normal modules iff}, if $f:M \to N$ is a normal epimorphism, then we may assume that $f:M \to M/L$ for some saturated submodule $L$ of $M$, where $f$ is a natural projection. In particular, one can easily check that $S(f):S(M) \to S(M/L)=S(M)/S(L)$ sends any saturated submodule $K$ to $\angles{K \cup L}$, i.e., $S(f)$ is given by joining with $L$ (as an element of $S(M)$). This shows that $S(f)$ is a normal epimorphism. 
\end{proof}

\begin{pro}\label{proposition: B-mdoule subcategory}
The functor $S:\emph{Mod}_\mathbb{B} \to \mathcal{L}$ is fully faithful and essentially surjective. Moreover, $S$ is an exact functor. 
\end{pro}
\begin{proof}
It follows from Theorem \ref{theorem: lattice B-mdoule thm} that $S$ is essentially surjective. Once we prove that $S$ is also fully faithful, it follows from Lemma \ref{lemma: admissible lemma} that the functor $S$ is exact. So, we only have to prove that $S$ is fully faithful.
	
Suppose that $S(f)=S(g)$ for $f,g \in \Hom(M,N)$ for $\mathbb{B}$-modules $M,N$. Note that for any $x \in M$ the saturated submodule of $M$ generated by $x$ is $\{m \in M \mid m \leq x\}$. In particular, this implies that if $S(f)=S(g)$, then $f=g$, showing that the functor $S$ is faithful. 

Next, consider $\alpha:S(M) \to S(N)$. This restricts to $\alpha^c:S(M)^c \to S(N)^c$ since $\alpha$ sends compacts elements to compact elements. With the isomorphism in Theorem \ref{theorem: lattice B-mdoule thm}, we have $S(M)^c\simeq M$ and $S(N)^c\simeq N$. One can easily see that  $\alpha^c \in \Hom(M,N)$ and $S(\alpha^c)=\alpha$, showing that the functor $S$ is full. 
\end{proof}

\begin{cor}\label{corollary: main theorem lattice and module}
Let $\emph{Mod}_\mathbb{B}^{\emph{fin}}$ be the category of finite $\mathbb{B}$-modules and $\mathcal{L}^c$ be the subcategory of $\mathcal{L}$ consisting of finite algebraic lattices. Then, $\emph{Mod}_\mathbb{B}^{\emph{fin}}$ is equivalent to $\mathcal{L}^c$ as proto-exact categories.
\end{cor}
\begin{proof}
The equivalence of categories in Proposition \ref{proposition: B-mdoule subcategory} restricts to the equivalence of the categories $\textrm{Mod}_\mathbb{B}^{\textrm{fin}}$ and $\mathcal{L}^c$. Hence, we have a quasi-inverse:
\[
S^{-1}:\mathcal{L}^c \to \textrm{Mod}_\mathbb{B}^{\textrm{fin}}.
\]
It is enough to prove that $S^{-1}$ is exact. In fact, we only have to prove that $S^{-1}$ sends normal monomorphisms (resp.~normal epimorphisms) in $\mathcal{L}^c$ to normal monomorphisms (resp.~normal epimorphisms) in $\textrm{Mod}_\mathbb{B}^{\textrm{fin}}$.

Suppose that $S(f):S(M) \to S(N)$ is a normal monomorphism of finite algebraic lattices. From Lemma \ref{lemma: normal modules iff}, it is enough to show that $f$ is an injection and $f(M)$ is isomorphic to a saturated submodule of $N$. It follows from Lemmas \ref{lemma: saturation closure} and \ref{lemma: vee lemma} that for any $x \in M$ we have
\begin{equation}\label{eq: exact}
S(f)(\angles{x})=\angles{f(x)} \textrm{ and } \angles{x}=\{y \in M \mid y \leq x\}.
\end{equation}
In particular, $f$ is an injection; if $\angles{f(x)}=\angles{f(y)}$, then since $S(f)$ is injective, we have $\angles{x}=\angles{y}$. This implies that $x \leq y$ and $y \leq x$ and hence $x=y$.

Furthermore, $f(M)$ is a saturated submodule of $N$. To see this, suppose that we have $x \in N$ and $y \in f(M)$ such that $x+y \in f(M)$. We have that $x \leq x+y$, hence $\angles{x} \leq \angles{x+y}$ in $S(N)$. But, since $S(f)$ is normal, the image $S(f)(S(M))$ is downward closed. It follows that $\angles{x}=S(f)(L)$ for some saturated submodule $L$ of $M$, in particular, $\angles{x}$ is the saturation closure of $f(L)$. This implies that there is some $q \in L$ such that $x \leq f(q)$. But, on the other hand, $f(q)$ is in $f(L)$ and hence is in $\angles{x}$, showing that $f(q) \leq x$. This implies that $x \in f(M)$.  Therefore $f(M)$ is a saturated submodule of $N$.

Next, suppose that $S(f):S(M) \to S(N)$ is a normal epimorphism of finite algebraic lattices. Write $S(N)=S(M)/L$ for some algebraic sub-lattice $L$. Let $H$ be the maximal element of $L$. We note that since $S(M)$ and $S(N)$ are finite algebraic lattices, any element is compact, and hence from Theorem \ref{theorem: lattice B-mdoule thm} any element of $S(M)$ (resp.~$S(N))$ corresponds to a unique element in $M$  (resp.~$N$).

We first claim that $f:M \to N$ is surjective. Let $y \in N$. Then, for the corresponding compact element $\angles{y} \in S(N)$, it follows from Lemma \ref{lemma: common lemma} that there is some compact element $\angles{x} \in S(M)$ such that $\angles{y}$ is the join of $H$ and $\angles{x}$. In other words, we have
\[
\angles{y} = H \vee \angles{x} =S(f)(\angles{x}) 
\] 
In particular, $\angles{y}=\angles{f(x)}$, and hence $y=f(x)$. This shows that $f$ is surjective.

Now, if $a \in H$, then $\angles{a}$ is a subset of $H$, hence is in $L$ since $L$ is downwards closed. This implies that $S(f)(\angles{a})=0$, and hence $f(a)=0$. In particular, $f$ factors through $M/H$. Suppose $x,y \in M$ satisfy $f(x)=f(y)$. Then we have $f(\angles{x})=f(\angles{y})$, i.e., $\angles{x,H}=\angles{y,H}$. Therefore, there exist $h_1,h_2 \in H$ such that 
\[
x\leq y+h_1, \quad y \leq x+h_2, 
\]
which implies that
\[
x+(h_1+h_2)=y+(h_1+h_2).
\]
It follows that $x$ and $y$ determine the same equivalence class in $M/H$, and hence the induced map $M/H \to N$ is injective, showing that $N=M/H$. This shows that $f$ is a normal epimorphism.
\end{proof}




\begin{myeg}
	Consider the following algebraic lattices. 
	\[
	L'=\left(\begin{tikzcd}[row sep=1em]	
		& \mycirc \arrow[dl,dash]  \arrow[dr,dash] \arrow[d,dash]\\
		\mycirc & \mycirc& \mycirc 
	\end{tikzcd}\right), \quad L''=\left(\begin{tikzcd}[row sep=1em]	
		& \mycirc \arrow[dl,dash]  \arrow[dr,dash]\\
		\mycirc \arrow[dr,dash]& & \mycirc \arrow[dl,dash]  \\
		& \mycirc & 
	\end{tikzcd}\right)
	\]
	The following are some possible algebraic lattices $L$ which fit into a short exact sequence $L' \to L \to L''$, where green dots belong to $L'$, the red dots belong to $L''$, and the blue dot belongs to both $L'$ and $L''$.
	\[
	L_1=\left(\begin{tikzcd}[row sep=1em]	
		& \mycirc[red] \arrow[dl,dash]  \arrow[dr,dash]\\
		\mycirc[red] \arrow[dr,dash]& & \mycirc[red] \arrow[dl,dash] \\
		& \mycirc[blue] \arrow[dl,dash]  \arrow[dr,dash] \arrow[d,dash]\\
		\mycirc[green] & \mycirc[green]& \mycirc[green] 
	\end{tikzcd}\right), ~ L_2= \left(\begin{tikzcd}[row sep=1em]	
		& \mycirc[red] \arrow[dl,dash]  \arrow[dr,dash]\\
		\mycirc[red] \arrow[dd,dash] \arrow[dr,dash]& & \mycirc[red] \arrow[dl,dash] \\
		& \mycirc[blue] \arrow[dl,dash]  \arrow[dr,dash] \arrow[d,dash]\\
		\mycirc[green] & \mycirc[green]& \mycirc[green] 
	\end{tikzcd}\right),~L_3= \left(\begin{tikzcd}[row sep=1em]	
		& \mycirc[red] \arrow[dl,dash]  \arrow[dr,dash]\\
		\mycirc[red] \arrow[dr,dash]& & \mycirc[red] \arrow[dl,dash] \arrow[dd,dash] \\
		& \mycirc[blue] \arrow[dl,dash]  \arrow[dr,dash] \arrow[d,dash]\\
		\mycirc[green] & \mycirc[green]& \mycirc[green] 
	\end{tikzcd}\right)
	\]
	\[
	L_4= \left(\begin{tikzcd}[row sep=1em]	
		& \mycirc[red] \arrow[dl,dash]  \arrow[dr,dash]\\
		\mycirc[red] \arrow[dr,dash] \arrow[dd,dash]& & \mycirc[red] \arrow[dl,dash] \arrow[dd,dash] \\
		& \mycirc[blue] \arrow[dl,dash]  \arrow[dr,dash] \arrow[d,dash]\\
		\mycirc[green] & \mycirc[green]& \mycirc[green] 
	\end{tikzcd}\right)
	\]
\end{myeg}

Even though the category $\mathcal{L}^c$ (or $\textrm{Mod}_\mathbb{B}^{\textrm{fin}}$) concerns only finite sets, surprisingly they are not finitary as the following simple example illustrates.

\begin{myeg}\label{example: pathology}
	Let $n$ be a positive integer, and $L_n=\{0,a_1,\dots,a_n,1\}$ be an algebraic lattice such that $0$ is the smallest element, $1$ is the largest element, and $a_i$ are incomparable. Then $K=\{0,a_1\}$ is an algebraic sub-lattice of $L_n$ which is isomorphic to $\mathbb{B}$. The quotient $L_n/K$ consists of elements greater than or equal to $a_1$, that is $\{a_1,1\}$. So, one has the following short exact sequence for any positive integer $n$:
\begin{equation}
	\ses{\mathbb{B}}{L_n}{\mathbb{B}}.
\end{equation}
In other words, we have
	\begin{equation}\label{eq: lattice example}
		L_n \iff \left(
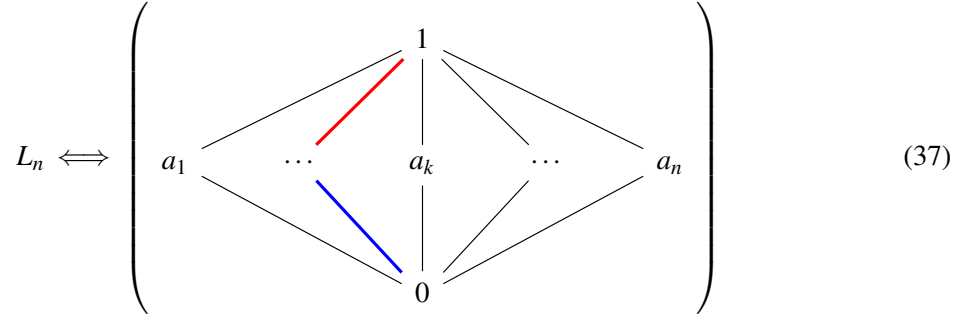
\begin{tikzcd}[row sep=1em]	
			& & 1  \arrow[dddll,dash]  \arrow[dddl,dash, red, very thick] \arrow[ddd,dash] \arrow[dddrr,dash]  \arrow[dddr,dash]& &\\
			& & & & \\
			& & & & \\
			a_1 \arrow[dddrr,dash] & \cdots \arrow[dddr,dash,blue, very thick]& a_k \arrow[ddd,dash]& \cdots\arrow[dddl,dash] & a_n \arrow[dddll,dash]\\
			& & & & \\
			& & & & \\
			&	&   0 & &
		\end{tikzcd}\right)
	\end{equation}
	and each $\mathbb{B}$ corresponds to either a red lattice or a blue lattice in \eqref{eq: lattice example}. Since for $n \neq m$ $L_n$ is not isomorphic to $L_m$, the set $\textrm{Ext}_{\mathcal{L}^c}(\mathbb{B},\mathbb{B})$ is not finite, showing that $\mathcal{L}^c$ is not finitary. 
\end{myeg}

Finally, we consider geometric lattices. Since geometric lattices are finite, they are algebraic lattices. In particular, the proto-exact structure of $\mathcal{L}$ may provide the induced proto-exact structure for geometric lattices. We first recall the definition of geometric lattices, which provides a cryptomorphic definition for simple matroids. For details, we refer the reader to \cite{oxley2006matroid}. 

\begin{mydef}
Let $P$ be a finite poset. 
\begin{enumerate}
	\item 
A \emph{chain} in $P$ from $x_0$ to $x_n$ is a subset $\{x_0,x_1,\dots,x_n\}$ of $P$ such that 
\[
x_0 < x_1< \cdots < x_n.
\]
The length of a chain $\{x_0,\dots,x_n\}$ is $n$. A chain is said to be maximal if there is no element $z$ such that $x_i < z <x_{i+1}$ for all $i=0,\dots,n-1$. 
\item
$P$ is said to satisfy the \emph{Jordan-Dedekind chain condition} if for any pair $\{x,y\}\subseteq P$ with $x<y$, all maximal chains from $x$ to $y$ have the same length. 
\end{enumerate}
\end{mydef}

Let $P$ be a poset with the minimal element $0$ satisfying the Jordan-Dedekind chain condition. One can define the \emph{height function}
\[
h:P \to \mathbb{Z}
\]
such that $h(x)$ is the length of a maximal chain from $0$ to $x$. An \emph{atom} of $P$ is an element of height $1$.

\begin{mydef}
Let $L$ be a finite lattice. 
\begin{enumerate}
	\item 
$L$ is said to be a \emph{semimodular lattice} if $L$ satisfies the Jordan-Dedekind chain condition, and the height function on $L$ satisfies the following condition: for $x,y \in L$, 
\[
h(x)+h(y) \geq h(x\vee y) + h (x \land y).
\]
\item 
$L$ is said to be a \emph{geometric lattice} if $L$ is semimodular in which every element is a join of atoms. 
\end{enumerate}
\end{mydef}

\begin{pro}\label{proposition: matroid geometric lattices}
Let $\mathcal{G}$ be the subcategory of $\mathcal{L}$ consisting of geometric lattices. Then, $\mathcal{G}$ is a proto-exact subcategory of $\mathcal{L}$ with the same proto-exact structure. 
\end{pro}
\begin{proof}
Clearly, the zero object (the empty lattice) is in $\mathcal{G}$. Hence, we only have to show that $\mathcal{G}$ is closed under taking subobjects and quotients. But, this is clear since any interval of a geometric lattice is again a geometric lattice. 
\end{proof}

\subsection{Further directions}\label{subsection: semiring further directions}

	In \cite{eppolito2018proto}, together with C.~Eppolito, the first and the second authors showed that the category of pointed matroids with strong maps is finitary and proto-exact, where admissible monomorphisms are strong maps that can be factored as
	\[
	N \overset{\sim}{\rightarrow} M \vert S \hookrightarrow M
	\]
	and admissible epimorphism are strong maps that can be factored as
	\[
	M \twoheadrightarrow M/S \overset{\sim}{\rightarrow} N
	\]
	In \cite[Section 2]{crapo1967structure}, H.~Crapo showed that a strong map between matroids corresponds to a function $f:L_1 \to L_2$ between geometric lattices satisfying the following three conditions:
	\begin{enumerate}
		\item 
		(non-singular) $f(x)=0_{L_2}$ implies $x=0_{L_1}$; 	
		\item 
		(join-homomorphism) $f(\sup X)=\sup f(X)$ for any $X \subseteq L_1$;	
		\item 
		(cover-preserving) The image of each atom in $L_1$ is either $0$ or an atom of $L_2$. 
	\end{enumerate}
	One can easily that Definition \ref{defnition: morphis of algebraic lattices} is equivalent to a function being a join-homomorphism for geometric lattices. In particular, our subcategory $\mathcal{G}$ of geometric lattices contains more morphisms than the category of geometric lattices with strong maps, and as a result we have a pathological examples such as Example \ref{example: pathology}. This leads us to the following question. 
	
\begin{question}
Is there a ``better'' notion of morphisms (as well as admissible morphisms) for algebraic lattices in such a way that the category of algebraic lattices is finitary and proto-exact? If so, when restricted to geometric lattices, do they define strong maps for matroids?
\end{question}

\section{The category of modules over a hyperring as a proto-exact category} \label{section: proto-exact hyperring}

In this section, we study the category of modules over a hyperring from the perspective of proto-exact categories. Throughout this section, let $H$ be a hyperring and $\Mod_H$ be the category of $H$-modules. We note that in \cite{madanshekaf2006exact}, certain categorical aspects of modules over a hyperring are studied, however it does not show that the category of modules over a hyperring is proto-exact.

\begin{mydef}\label{definition: proto-exact hyperring}
Let $\mathfrak{M}$ be the class of strict injective homomorphisms in $\Mod_H$ and $\mathfrak{E}$ be the class of strict surjective homomorphisms in $\Mod_H$. 
\end{mydef}

In the following, we show that $(\Mod_H,\mathfrak{M},\mathfrak{E})$ is a proto-exact category. It is clear that $\mathfrak{M}$ and $\mathfrak{E}$ contain all isomorphisms and are closed under composition. 

Let $B$ be an $H$-module, and $A$ a submodule of $B$. Then, $A$ defines an equivalence relation $\equiv$ on $B$ as follows: for $b_1,b_2 \in B$, 
\begin{equation}\label{eq: submoudle congruence}
b_1 \equiv b_2 \iff b_1+A = b_2+A,
\end{equation}
where $b_i+A=\bigcup_{a \in A} (b_i+a)$ and $b_1+A=b_2+A$ is equality of sets. 

\begin{mydef}\label{definition: two sets equiv}
Let $M$ be an $H$-module and $\equiv$ be an equivalence relation on $M$. For subsets $X,Y\subseteq M$, we write $X \equiv Y$ if the following two conditions holds:
\begin{enumerate}
	\item 
For any $x \in X$, there exists $y \in Y$ such that $x \equiv y$. 
\item 
For any $y \in Y$, there exists $x \in X$ such that $x \equiv y$. 
\end{enumerate}
\end{mydef}

\begin{mydef}
By a congruence relation on $M$, we mean an equivalence relation $\equiv$ on $M$ such that for any $x_i,y_i \in M$ for $i=1,2$ and $r \in H$, the following holds:
\[
\textrm{If } x_i \equiv y_i \textrm{ for $i=1,2$, then, $(x_1+x_2)\equiv (y_1+y_2)$ and $(rx_i)\equiv (ry_i)$,}
\]
where the notation $(x_1+x_2)\equiv (y_1+y_2)$ is as in Definition \ref{definition: two sets equiv}.
\end{mydef}

\begin{pro}\label{proposition: quotient}
Let $B$ be an $H$-module and $A$ be a submodule of $B$. Then, $B/\equiv$ (with $\equiv$ as in \eqref{eq: submoudle congruence}) is an $H$-module, with the following hyperaddition:
\[
[x]+[y]:=\{[z] \mid z \in x'+y' \emph{ such that } [x']=[x],[y']=[y]\}, 
\]
and an $H$-action $r[x]:=[rx]$, where $r[x]=\{rx' \mid x'\in [x]\}$. 
\end{pro}
\begin{proof}
One may apply a similar argument as in \cite[Proposition 3.16]{jun2018algebraic} to prove that $B/\equiv$ is a hypergroup. 	
	
Next, we claim that $H$-action is well-defined, i.e., if $[x']=[x]$, then $[rx']=[rx]$. In fact, if $[x']=[x]$, then we have
\[
x'+A=x+A, 
\]
in particular, $x' \in x+a$ for some $a \in A$. It follows that for any $r \in H$, we have $rx' \in rx+ra$, showing that  $rx'+A \subseteq rx+A$ since $ra \in A$. A similar argument shows that $rx+A \subseteq rx'+A$, and hence $[rx']=[rx]$.

Finally, it is clear that $0[x]=[0x]=[0]$ and $1[x]=[x]$. Suppose that $[m],[n] \in B/\equiv$. Then, we have
\begin{equation}\label{eq: strict}
[m]+[n]=[m+n]. 
\end{equation}
In fact, it is clear from the definition that $[m+n] \subseteq [m]+[n]$. Conversely, if $[x] \in [m]+[n]$, then we have $x \in m'+n'$ for some $[m']=[m]$ and $[n']=[n]$. But, since  $[m']=[m]$ and $[n']=[n]$, there exist $a_1,a_2 \in A$ such that
\[
m' \in m+a_1, \quad n' \in n+a_2.
\]
It follows that $x \in (m+n)+A$. Hence $[x] \in [m+n]$, showing that $[m]+[n]\subseteq [m+n]$. This implies that for any $r \in H$, we have $r([m]+[n])=r[m+n]=[rm+rn]=r[m]+r[n]$. Furthermore, for $r,s \in R$ and $[x] \in B/\equiv$ we have
\[
(r+s)[m]=[(r+s)m]=[rm+sm]=[rm]+[sm]=r[m]+s[m].
\] 
This shows that $B/\equiv$ is an $H$-module. 
\end{proof}

In what follows, we will denote the $H$-module $(B/\equiv)$ as in Proposition \ref{proposition: quotient} by $B/A$.

\begin{cor}\label{corollary: strict surjective quotient}
Let $B$ be an $H$-module and $A$ be a submodule of $B$. Then, the canonical map $\pi:B \to B/A$ sending $x$ to $[x]$ is a strict surjective homomorphism of $H$-modules. 
\end{cor}
\begin{proof}
This is proved in Proposition \ref{proposition: quotient} in \eqref{eq: strict}. 
\end{proof}

\begin{lem}\label{lemma: bi-cat}
Let $(\Mod_H,\MM,\EE)$ be as in Definition \ref{definition: proto-exact hyperring}. Let $i,i' \in \mathfrak{M}$ and $j,j' \in \mathfrak{E}$, $A$ an $H$-module, and $C$ a submodule of $A$. Then the following commutative diagram is a bi-Cartesian diagram in $\Mod_H$.
\begin{equation}\label{eq: Bicartesian}
\begin{tikzcd}
A
\ar[r,"i",hook]
\ar[d,"j",two heads,swap]
&
B
\ar[d,"j'",two heads]
\\
A/C
\ar[r,"i'",hook]
&
B/C
\end{tikzcd}
\end{equation}
\end{lem}
\begin{proof}
We first note that since $C$ is a submodule of $A$ and $A$ is a submodule of $B$, $C$ is also a submodule of $B$. In particular the quotient $B/C$ is well-defined. 

Next, we claim that the induced map $i':A/C \to B/C$ sending $[x]$ to $[i(x)]$ is well-defined. Indeed, if $[x]=[x']$, then we have $x+C=x'+C$, in particular, $x' \in x+c$ for some $c \in C$. It follows that $i(x') \in i(x) +i(c)$, hence $i(x')+C \subseteq i(x)+C$. By the same argument, one sees that $i(x)+C \subseteq i(x')+C$, showing that $[i(x)]=[i(x')]$. Hence $i'$ is well-defined. Furthermore $i'$ is an injection since for $[i(x)]=[i(y)]$, we have $i(x) \in i(y)+i(c)=i(y+c)$ for some $c \in C$ implying $[x]=[y]$. Finally, $i'$ is strict since, by Corollary \ref{corollary: strict surjective quotient}, 
\[
i'([x]+[y])=i'([x+y])=[i(x+y)]=[i(x)+i(y)]=[i(x)]+[i(y)].
\]
This shows that $i'$ is an admissible monomorphism. In particular, the diagram \eqref{eq: Bicartesian} is well-defined. 

(co-Cartesian) Suppose that we have the following commutative diagram with an $H$-module $D$ and $H$-module homomorphisms $\alpha$ and $\beta$:
 \begin{equation}\label{eq: pushout}
\begin{tikzcd}
A \ar[hook]{r}{i} \ar[swap,two heads]{d}{\pi} & B \ar[swap,two heads]{d}{\pi'}
\ar[bend left]{ddr}{\beta} &
\\
A/C \ar[hook]{r}{i'} \ar[swap, bend right]{rrd}{\alpha} & B/C \ar[dotted]{dr}{\gamma} &
\\
& & D
\end{tikzcd}
\end{equation}
We define $\gamma:B/C \to D$ by $\gamma([x])=\beta(x)$. We claim that $\gamma$ is well-defined. Indeed, if $[x]=[y]$ in $B/C$, then we have $x+C=y+C$, hence $x \in y+n$ for some $n \in C$. From the commutativity of the diagram \eqref{eq: pushout}, we have
\[
\beta(n)=\beta(i(n))= \alpha(\pi(n))=0.
\]
In particular, we have
\[
\beta(x) \in \beta(y+n) \subseteq \beta(y)+\beta(n)=\beta(y),
\]
showing that $\beta(x)=\beta(y)$, and hence $\gamma([x])=\gamma([y])$. Next, we claim that $\gamma$ is a homomorphism of $R$-modules. For $r \in R$, we have
\[
\gamma(r[x])=\gamma([rx])=\beta(rx)=r\beta(x)=r\gamma([x]). 
\]
Also, suppose that $[x],[y] \in B/C$ and $[z] \in [x]+[y]$. In other words, $z \in x'+y'$ for some $[x']=[x]$ and $[y']=[y]$ in $B/C$. It follows that
\[
\gamma([z])=\beta(z) \in \beta(x')+\beta(y')=\beta(x)+\beta(y)=\gamma([x])+\gamma([y]), 
\]
showing that $\gamma$ is a homomorphism of $H$-modules. It is clear that with $\gamma$, \eqref{eq: pushout} commutes. Furthermore, one can easily check that if a morphism $\delta:B/C \to D$ makes \eqref{eq: pushout} to commutes, then $\delta=\gamma$. Hence $\eqref{eq: Bicartesian}$ is a co-Cartesian diagram. 

(Cartesian) Now, suppose that we have the following commutative diagram:
\begin{equation}\label{eq: Cat1}
 \begin{tikzcd}
D \ar[bend left]{drr}{\beta} \ar[swap, bend right]{ddr}{\alpha} \ar[dotted]{dr}{\gamma} & &
\\
&A\ar[hook]{r}{i} \ar[two heads]{d}{\pi} & B \ar[two heads]{d}{\pi'}
\\
& A/C \ar[swap,hook]{r}{i'} & B/C
\end{tikzcd}
\end{equation}
We claim that $\beta(D) \subseteq i(A)$. In fact, for any $x \in D$, we have $i'\alpha(x)=\pi'\beta(x)$. Since $\pi$ is surjective, we may write $\alpha(x)=\pi(a)$ for some $a \in A$. Therefore, we have
\[
\pi'i(a)=i'\pi(a)=\pi'\beta(x). 
\]
It follows that $i(a)+C=\beta(x)+C$, in particular, $\beta(x) \in i(a)+c$ for some $c \in C$. Since $C$ is a submodule of $A$, we have that $\beta(x) \in i(A)$, and hence $\beta(D) \subseteq i(A)$. 

From the above claim, the function $\gamma: D\to A$ sending $x \in D$ to $\beta(x)$ is well-defined (here we identify $A$ with $i(A)$). It is clear that $\gamma$ is a homomorphism of $H$-modules. Furthermore, with $\gamma$, the diagram \eqref{eq: Cat1} commutes. Finally, if we have $\delta:D \to A$ making \eqref{eq: Cat1} to commutes, then we have $i(\delta(x))=\beta(x)$, and hence $\delta(x)=\gamma(x)$ for all $x \in D$, showing that \eqref{eq: Cat1} is a Cartesian diagram. 
\end{proof}

\begin{lem}\label{lemma: submodule}
Let $A$ and $B$ be $H$-modules, and $f: A \to B$ be a strict homomorphism of $H$-modules. If $C$ is a submodule of $B$, then $f^{-1}(C)$ is a submodule of $A$. 
\end{lem}
\begin{proof}
Let $D=f^{-1}(C)$. We first claim that if $x,y \in D$, then $x+y \subseteq D$. In fact, let $x \in f^{-1}(a)$ and $y \in f^{-1}(b)$ for some $a,b \in C$. Suppose that $z \in x+y$. Then we have
\[
f(z) \in f(x)+f(y) = a+b \subseteq C 
\]
showing that $z \in D$.

Next, for $r \in H$ and $x \in D$ we have $rx \in D$. In fact, we write $f(x)=a \in C$, then we have $f(rx)=rf(x)=ra \in C$. Hence $rx \in f^{-1}(C)=D$. Now, it is clear that $D$ is a submodule of $A$. 
\end{proof}

\begin{lem}\label{lemma: first iso}
Let $j':B \to B'$ be a strict surjective homomorphism of $H$-modules. Then, $B'$ is isomorphic to $B/\ker(j')$.
\end{lem}
\begin{proof}
This follows from the first isomorphism theorem of $H$-modules which holds when one has strict homomorphisms. 
\end{proof}

\begin{lem}\label{lemma: pro4 hyper}
Every diagram in $\Mod_H$ of the following form
	\begin{equation}\label{eq: pro4 hyper}
		\begin{tikzcd}
			{} & B\ar["j'",d,two heads] \\
			A' \ar["i'",r,hook,swap] & B' 
		\end{tikzcd}
	\end{equation}
	with $i' \in \MM$ and $j' \in \EE$ can be completed to a
	bi-Cartesian square \eqref{eq: bi-cat} with $i \in \MM$ and
	$j \in \EE$.
\end{lem}
\begin{proof}
Let $A= (j')^{-1}(i'(A'))$. Since $i'(A')$ is a submodule of $B'$, it follows from Lemma \ref{lemma: submodule}, $A$ is a submodule of $B$. We claim that the inclusion $i:A \to B$ is a strict homomorphism. For us to show this, it is enough that show that for any $a, b \in A$, we have $a+b \subseteq A$, where $a+b$ is computed in $B$.\footnote{This is equivalent to $i(a+b)=i(a)+i(b)$.} Let $c \in a+b$. Then, there exist $x,y \in A'$ such that $j'(a)=i'(x)$ and $j'(b)=i'(y)$. In particular, since $j'$ and $i'$ are strict, we have
\[
j'(c) \in j'(a)+j'(b)=i'(x)+i'(y)=i'(x+y).
\]
It follows that $j'(c) \in i'(A')$, and hence $c \in A$. This implies that $i:A \to B$ is a strict injective homomorphism, and hence $i \in \MM$.

Next, consider the function $\pi:A \to A'$ sending any $a \in A$ to $((i')^{-1}j'i)(a)$. Note that $\pi$ is well-defined since $j'i(A) \subseteq i'(A')$ by definition. We claim that $\pi$ is a strict surjective homomorphism (admissible epimorphism). First, one can easily see that $\pi$ is a homomorphism. To show that $\pi$ is strict, suppose that $a,b \in A$ and $d \in \pi(a)+\pi(b)$. Since $i',i,j'$ are strict, this implies that
\[
i'(d) \in j'i(a)+j'i(b)=j'i(a+b).
\]
In particular, $d \in (i')^{-1}j'i(a+b)=\pi(a+b)$, showing that $\pi$ is strict. Finally, it is clear from the definition that $\pi$ is surjective, showing that $\pi$ is an admissible epimorphism.

Moreover, since $j':B \to B'$ is an admissible epimorphism, from Lemma \ref{lemma: first iso}, we may assume that $j':B \to B/C$, where $C$ is a submodule of $B$. Then, by definition of $A$, we have $C\subseteq A$ (we identify $i(A)$ and $A$). In particular, we may assume that
\[
\pi:A \to A'=A/C.  
\]
Therefore, we can complete the diagram \eqref{eq: pro4 hyper} as follows:
\begin{equation}\label{eq: Bicar hyper}
	\begin{tikzcd}
		A
		\ar[r,"i",hook]
		\ar[d,"\pi",two heads,swap]
		&
		B
		\ar[d,"j'",two heads]
		\\
		A/C
		\ar[r,"i'",hook]
		&
		B/C
	\end{tikzcd}
\end{equation}
Now it follows from Lemma \ref{lemma: bi-cat} that \eqref{eq: Bicar hyper} is a bi-Cartesian diagram.
\end{proof}

\begin{lem}\label{lemma: eq5 hyper}
Every diagram in $\Mod_H$ of the following form
	\begin{equation}\label{eq: pro5 hyper}
		\begin{tikzcd}
			A \ar["j",d,two heads,swap] \ar["i",r,hook] & B\\
			A'  & {}
		\end{tikzcd}
	\end{equation}
	with $i \in \MM$ and $j \in \EE$ can be completed to a
	bi-Cartesian square \eqref{eq: bi-cat} with $i' \in \MM$ and
	$j' \in \EE$.
\end{lem}
\begin{proof}
Since $j$ is an admissible epimorphism, from Lemma \ref{lemma: first iso}, we may assume that $j:A \to A'=A/C$ for some submodule $C$ of $A$. We let $j':B \to B/C$ be the map sending $x$ to $[x]$, which is an admissible epimorphism by Corollary \ref{corollary: strict surjective quotient}. Moreover, $i$ induces a natural map
\[
i': A/C \to B/C, \quad [x] \mapsto [i(x)].
\]
From the proof of Lemma \ref{lemma: bi-cat}, we know that $i'$ is well-defined and $i' \in \MM$. In particular, we can complete the diagram \eqref{eq: pro5 hyper} as follows:
\begin{equation}\label{eq: propro hyper}
	\begin{tikzcd}
		A
		\ar[r,"i",hook]
		\ar[d,"j",two heads,swap]
		&
		B
		\ar[d,"j'",two heads]
		\\
		A/C
		\ar[r,"i'",hook]
		&
		B/C
	\end{tikzcd}
\end{equation}
Now, it follows from Lemma \ref{lemma: bi-cat} that \eqref{eq: propro hyper} is a bi-Cartesian diagram. 
\end{proof}

\begin{mythm}\label{theorem: main theorem for hyperrngs}
The triple $(\Mod_H,\mathfrak{M},\mathfrak{E})$ is a proto-exact category. 
\end{mythm}
\begin{proof}
	$\Mod_H$ is clearly pointed. The axiom $(2)$ in Definition \ref{definition: proto_exact} is clear. The axioms $(4)$ and $(5)$ in Definition \ref{definition: proto_exact} directly follow from Lemmas \ref{lemma: pro4 hyper} and \ref{lemma: eq5 hyper}. Finally, the axiom $(3)$ is a direct consequence of $(4)$ and $(5)$ along with the uniqueness of pullbacks and pushouts. 
\end{proof}

\subsection{Finite $\mathbb{K}$-modules and incidence geometries}

Let $\mathbb{K}$ be the Krasner hyperfield. One can easily check that a hypergroup $E$ is an $\mathbb{K}$-module if and only if $x+x=\{0,x\}$ for all $x \neq 0$.  In \cite{con3}, Connes and Consani proved that each finite $\mathbb{K}$-module $E$ uniquely determines a projective geometry $\mathcal{P}_E$ whose set of points is $E\setminus \{0_E\}$ and the line $\ell(x,y)$ passing through two distinct points $x,y \in \mathcal{P}$ is defined by using the hyperaddition of $E$. To be a bit more precise, the points on the line $\ell(x,y)$ is given as follows:
\[
\ell(x,y)=(x+y) \cup \{x,y\}. 
\]
In this case, any line in $\mathcal{P}_E$ contains at least four points. Conversely, any projective geometry such that each line contains at least four points uniquely arises in this way.

Let $\Mod^{\textrm{fin}}_{\mathbb{K}}$ be the category of finite $\mathbb{K}$-modules. Let $E$ be a finite $\mathbb{K}$-module, and $\mathcal{P}_E$ the corresponding projective geometry. Let $\mathcal{L}_E$ be the lattice of submodules $E$. Then one can easily observe the following:
\begin{enumerate}
\item 
The atoms of $\mathcal{L}_E$ have the form $\{0,x\}$ for $x \neq 0$, and are in one-to-one correspondence with the points of $\mathcal{P}_E$. 
\item 
A submodule of height 2\footnote{By this we mean minimal among those which properly contain an atom} must contain $\{0,x\}$ for some $x$ and must contain $y$ for some nonzero $y \neq x$. Thus such a submodule must contain $x+y$, so contains all points of the line passing through $x$ and $y$. Since such a line (together with $0)$ is a submodule, a submodule of height 2 is a line in $\mathcal{P}_E$. Moreover, since a line cannot properly contain another line, all lines in $\mathcal{P}_E$ have height 2. 
\item 
Consider a short exact sequence\footnote{Recall that by a short exact sequence, we mean a commutative diagram of the form \eqref{eq: ses proto-exact} with the proto-exact structure of $\Mod^{\textrm{fin}}_{\mathbb{K}}$ induced from $\Mod_{\mathbb{K}}$.} in $\Mod^{\textrm{fin}}_{\mathbb{K}}$ as follows:
\begin{equation}\label{eq: ses}
\ses{E'}{E}{E''}.
\end{equation}
Recall that the submodules of $E/E' (\simeq E'')$ are in one-to-one correspondence with those of $E$ which contain $E'$ as in the classical case. In particular, a point in $E''$ is a submodule (or projective subspace) which is minimal among those which properly contain $E'$, while a line is minimal among submodules which are not points of $E''$ but properly contain $E'$, and so on. For example, when $E'=\{0,x\}$ (or a point in the corresponding projective geometry), the projective geometry $\mathcal{P}_{E''}$ is the quotient geometry as in \cite[Section 1.4]{beutelspacher1998projective}. 
\item 
With the short exact sequence \ref{eq: ses}, in the special case where $E$ is the projective space associated to a vector space $V$ over some field, and $E'$ corresponds to a subspace $W$ of $V$, $E''$ corresponds to the projectivization of $V/W$. 
\end{enumerate}

From the above, one may be tempted to use this correspondence to define the Hall algebra of $\Mod^{\textrm{fin}}_{\mathbb{K}}$ as it appears to be finitary. However, the following example shows that unfortunately $\Mod^{\textrm{fin}}_{\mathbb{K}}$ is not finitary. 

\begin{myeg}\label{eg: not finitary}
Let $n\geq 3$, and $E_n=\{0,a_1,\dots,a_n\}$. We impose hyperaddition on $E$ as follows:
\[
a_i + a_j = \begin{cases}
\{0, a_i\} \textrm{ if } i=j, \\
 \{a_k \mid k\neq j \textrm{ and } k \neq i \}  \textrm{ if } i \neq j. 	
\end{cases}
\]	
One can think of $\mathcal{P}_{E_n}$ as a projective space with only a single line but many points.  If we consider the quotient projective space by $a_1$ as in the above observation (3), the points of the quotient are lines through $a_1$, of which there are only one. In particular, there is an exact sequence $\ses{\mathbb{K}}{E_n}{\mathbb{K}}$ for each $n \geq 3$. This shows that $\Mod^{\textrm{fin}}_{\mathbb{K}}$ is not finitary since if $n \neq m$, then $E_n$ are $E_m$ are not isomorphic. In fact, one may apply a similar idea to show that $\Mod^{\textrm{fin}}_{\mathbb{K}}$ is not closed under extensions. 
\end{myeg}


\subsection{Further directions}

Let $H_\mathbb{K}$ be the vector space over complex numbers spanned by the set of isomorphism class in $\Mod^{\textrm{fin}}_{\mathbb{K}}$. Let $E_1,E_2$ be finite $\mathbb{K}$-modules. Suppose that in $H_\mathbb{K}$ we have
\begin{equation}\label{eq: not fin}
E_1\cdot E_2:= \sum_{E \in \textrm{Iso}(\Mod^{\textrm{fin}}_{\mathbb{K}})} \textbf{a}_{E_1,E_2}^EE, 
\end{equation}
where 
\begin{equation}
\textbf{a}_{E_1,E_2}^E=\#\{L \subseteq E \mid L \simeq E_2 \textrm{ and } E/L \simeq E_1\}.
\end{equation}
From the above discussion we see that the number $\textbf{a}_{E_1,E_2}^E$ counts the number of different embeddings of projective geometries $L$, which is isomorphic to $\mathcal{P}_{E_2}$, into $\mathcal{P}_E$ such that the ``quotient geometry'' $\mathcal{P}_{E/L}$ is isomorphic to $E_1$. Here is a more explicit example. 
 
 \begin{myeg}
Let $E_2=\mathbb{K}$ (as a hypergroup) and Let $E_1=\mathbb{F}_{5^2}/\mathbb{F}_5^\times$. Then $\mathcal{P}_{E_2}$ is just a point and $\mathcal{P}_{E_1}=\mathbb{P}^1_{\mathbb{F}_5}$, the projective line over $\mathbb{F}_5$. Hence, when $E=\mathbb{P}^2_{\mathbb{F}_5}$, the structure constant $\textbf{a}_{E_1,E_2}^E$ is the number of flags $(a,\ell)$ in $\mathbb{P}^2_{\mathbb{F}_5}$, where $a$ is a point and $\ell$ is a line containing $a$.  
 \end{myeg}

\begin{rmk}\label{rmk: functorial}
In fact, the aforementioned correspondence between finite $\mathbb{K}$-module and projective geometries are functorial, where the category of projective geometries is defined as in \cite{faure1994morphisms}. Moreover, when $H=\mathbb{S}$, the sign hyperfield, finite $\mathbb{S}$-modules correspond to \emph{spherical geometries} and the above discussions are carried over to this case. 
\end{rmk}

As Example \ref{eg: not finitary} shows, the category $\Mod^{\textrm{fin}}_{\mathbb{K}}$ is not finitary, in particular, the multiplication \eqref{eq: not fin} does not have to be a finite sum in general. This leads us to the following questions.

\begin{question}
\begin{enumerate}
	\item 
Can one find a proto-exact subcategory $\mathcal{C}$ of $\Mod^{\textrm{fin}}_{\mathbb{K}}$ which is finitary?
\item 
If so, can one describe the subcategory of projective geometries through the categorical correspondence mentioned in Remark \ref{rmk: functorial}. In this case, do structure constants for $H_\mathcal{C}$ encode some interesting combinatorial identities for projective geometries? 
\end{enumerate}
 
\end{question}



\bibliography{hyperfield}\bibliographystyle{alpha}

\end{document}